\documentclass{elsarticle}

\pdfoutput=1
\usepackage[latin1]{inputenc}
\usepackage[all,cmtip]{xy}
\usepackage{array}
\usepackage{amsmath, amssymb}
\usepackage{graphics}
\usepackage{rotating}
\usepackage{caption}
\usepackage{subcaption}
\usepackage{amsthm}
\usepackage{algorithm}
\usepackage{algorithmic}
\usepackage{booktabs}
\usepackage{color}
\usepackage{xcolor}
\usepackage{cancel}
\usepackage{caption}
\usepackage{subcaption}
\usepackage{hhline}
\usepackage{pdflscape}
\usepackage{epstopdf}
\usepackage{subcaption}
\usepackage{multirow}
\usepackage{graphicx}
\usepackage[hidelinks]{hyperref}
\usepackage{multicol}
\usepackage{float}
\usepackage[top=4.4cm,left=4.7cm, right=4.75cm, bottom=4.2cm]{geometry}

\def\deg{{\rm deg}}

\newtheorem{theorem}{{\bf Theorem}}
\newtheorem{remark}{{\bf Remark}}

\newtheorem{corollary}[theorem]{{\bf Corollary}}
\newtheorem{proposition}[theorem]{{\bf Proposition}}
\newtheorem{lemma}[theorem]{{\bf Lemma}}

\def\bfb{{\boldsymbol{b}}}

\def\bfp{{\boldsymbol{p}}}
\def\bfq{{\boldsymbol{q}}}

\def\bfx{{\boldsymbol{x}}}
\def\bfy{{\boldsymbol{y}}}

\def\bfQ{{\boldsymbol{Q}}}
\def\bfA{{\boldsymbol{A}}}

\def\RR{\mathbb{R}}

\begin{document}

\begin{frontmatter}
\title{Affine equivalences, isometries and symmetries of ruled rational surfaces.}


\author[a]{Juan Gerardo Alc\'azar\fnref{proy,proy2}}
\ead{juange.alcazar@uah.es}
\author[a]{Emily Quintero\fnref{proy3}}
\ead{emily.quintero@edu.uah.es}

\address[a]{Departamento de F\'{\i}sica y Matem\'aticas, Universidad de Alcal\'a,
E-28871 Madrid, Spain}


\fntext[proy]{Supported by the Spanish ``Ministerio de
 Econom\'ia y Competitividad" under the Project MTM2014-54141-P.}

\fntext[proy2]{Member of the Research Group {\sc asynacs} (Ref. {\sc ccee2011/r34}) }

\fntext[proy3]{Supported by a grant from the Carolina Foundation.}


\begin{abstract}
We present a method for computing all the affine equivalences between two rational ruled surfaces defined by rational parametrizations that works directly in parametric rational form, i.e. without computing or making use of the implicit equation of the surface. The method proceeds by translating the problem into the parameter space, and relies on polynomial system solving. From a geometric point of view, an important observation is the fact that the problem is related to finding the projective equivalences between two projective curves (corresponding to the directions of the rulings of the surfaces). This problem was recently addressed by Hauer and J\"uttler in \cite{HJ18}, and the ideas by these authors are greatly exploited in the algorithm presented in this paper. The general idea is adapted to compute the isometries between two rational ruled surfaces, and the symmetries of a given rational ruled surface. The efficiency of the method is shown through several examples. 
\end{abstract}

 \end{frontmatter}

\section{Introduction}\label{section-introduction}

Two given surfaces are \emph{affinely equivalent} when one of them is the result of applying a nonsingular affine transformation to the other one. Any nonsingular affine transformation is a global diffeomorphism, and therefore the transformation preserves both the topology and the differential properties of the surfaces. One can say that in this case the surfaces are smooth deformations of each other. Recognizing affine equivalence is important in fields like Computer Vision or Pattern Recognition, where one often has objects stored in a database, and needs to recognize those objects up to some transformation. 

While affine transformations do not, in general, preserve metric properties, \emph{isometries} do. Isometries, which are symmetries of 3-space, are well classified \cite{Coxeter}, and comprise translations, central symmetries, reflections in a plane, rotational symmetries (with axial symmetries as a special case), and their composites. A \emph{symmetry} of a surface is a symmetry of 3-space that leaves the surface invariant. In particular, symmetries of 3-space are orthogonal transformations. Thus, two surfaces are isometric when one of them is the result of applying a rigid motion to the other one.

Additionally, knowing the symmetries of a surface is useful in order to understand the geometry of the surface and to visualize the surface correctly. It is also useful in applications like image storage and medial axis computations, or, again, object detection and recognition. In the literature of applied fields like Computer Aided Geometric Design, Pattern Recognition or Computer Vision one can find many methods to detect symmetries (see for instance the Introduction to \cite{AH15}), although these methods are usually applied to objects where no specific structure is assumed, and are more orientented towards finding approximate symmetries. 

The same thing can be said about methods to identify affine equivalences; see for example the paper \cite{Shahed} and the references provided therein. In fact, in applications the problem which has received more attention is the detection of affine equivalences between point clouds, since images and objects are often considered this way. 

In contrast, in this paper we address a type of surfaces with a strong structure, namely rational ruled algebraic surfaces, and we make use of the structure of the surfaces in order to compute affine equivalences, isometries or symmetries. Ruled surfaces consist of straight lines, and are classical in Differential and Algebraic Geometry. A complete account of many properties of these surfaces is given, for instance, in the books \cite{G06} and \cite{Pottmann}. 

Some recent publications address similar problems for curves and surfaces, too. Projective and affine equivalences between rational curves in arbitrary dimension are considered in \cite{HJ18}. The same problem for rational and polynomial surfaces is considered in \cite{HJ18-2}. However, in \cite{HJ18-2} the parametrizations of the surfaces involved are assumed to have no projective base points, which is not the case, in general, for rational ruled surfaces. Projective equivalences between some special varieties are also studied in \cite{BLV18}. In \cite{BLV18} ruled surfaces, among others, are considered, and an algorithm to compute projective equivalences, based on a strong background of Algebraic Geometry, is provided. Nevertheless, the authors of \cite{BLV18} are not very specific about computational questions or timings, so it is not easy to compare their methods with the one in this paper. 

A naive approach to solve the problem could be to compute first the implicit equations of the surfaces, which can be efficiently done by using $\mu$-bases \cite{Chen1, Chen2}, to pick a generic affine transformation, and finally to impose that the transformation maps one surface onto the other, which is equivalent to recovering, up to a constant, the implicit equation of the second surface. This approach leads to a polynomial system with 12 variables, the parameters of the affine transformation, where the degree of each polynomial equation is bounded by the degree of the implicit equation. However, this approach is often impractical. On one hand, the system has many variables, which makes it complicated to solve even when the polynomial equations have low total degree. On the other hand, computing the implicit equation may destroy sparsity, when present in the parametrization. The same criticism applies to computing isometries or symmetries by this method. 

In this paper we use a different approach, based on previous work of one of the authors on symmetries of rational curves \cite{AHM15} and polynomially parametrized surfaces \cite{AH15}. The idea is to reduce the problem to computations in the parameter space, an idea also present in \cite{HJ18} or \cite{HJ18-2}. In order to do this, we observe that whenever the parametrization defining the surfaces are \emph{proper}, i.e. birational, any affine equivalence of the surfaces is associated with a birational transformation of the parameter space (the plane), i.e. a Cremona transformation. Taking advantage of the structure of the surfaces, in our case of the fact that the surfaces are ruled, we prove that the corresponding Cremona transformation has a predictable form. 

This predictable form reveals something else, namely the fact that the matrix associated with an affine equivalence corresponds to a projective equivalence between the projective curves defined by the directions of the rulings of the surfaces. From this point, we take advantage of the detailed analysis of the computation of projective equivalences between curves carried out in \cite{HJ18}. In particular, in our case the main difficulty, from a computational point of view, is the solution of a polynomial system in the parameters of the Cremona transformation whose structure is studied in \cite{HJ18}; the core of the computational method we propose to find the affine equivalences of the surfaces is based on ideas in \cite{HJ18}.

For isometries or symmetries, the fact that the transformation we are seeking preserves metric properties provides extra conditions that help reduce the computation time. 
 
The structure of this paper is the following. We start with a preliminary section, Section \ref{gen-surf}, where we fix the hypotheses required on the input, together with some general notions to be used later in this paper. Several results behind the method, and the analysis of the structure of the Cremona transformation behind affine equivalences, are presented in Section \ref{sec-symmetries}. The computational method itself is addressed in Section \ref{sec-computation}, where we apply many results in \cite{HJ18}. 
We report on experiments in Section \ref{sec-exp}. Some brief observations on projective equivalences are provided in Section \ref{sec-projective}. A brief summary of our work is presented in Section \ref{sec-conclusion}.

\vspace{0.2 cm}
\noindent\noindent{\bf Acknowledgements.} Juan G. Alc\'azar is supported by the Spanish Ministerio de Econom\'{\i}a y Competitividad and by the European Regional Development Fund (ERDF), under the project  MTM2017-88796-P. Emily Quintero is supported by a grant from the Carolina Foundation. The authors are grateful to the reviewers, whose comments allowed to improve an earlier version of the paper. Particularly, one of the reviewers pointed out the similarities with the paper \cite{HJ18}, and greatly helped to improve the computation strategy and the timings, compared to those of an initial version of the paper, with very precise suggestions.

\section{Preliminaries.} \label{gen-surf}

Let $S_1,S_2$ be real ruled surfaces, defined by means of real, rational parametrizations $\bfx_1,\bfx_2$ in standard form \cite{SPD14}
\begin{equation}\label{surf}
\pmb{x}_i(t,s)=\pmb{p}_i(t)+s\cdotp\pmb{q}_i(t).
\end{equation}
An algorithm for writing a ruled rational surface in standard form is provided in \cite{SPD14}; in \cite{SPD14} it is shown that any ruled rational surface can be brought into an equation like Eq. \eqref{surf}, although the parametrization might not be real (e.g. quadrics). At each point $P=\pmb{x}_i(t,s)\in S_i$, $i=1,2$, the vector $\pmb{q}_i(t)$ defines the direction of the \emph{ruling} through $P$, i.e. the line $L_P$ through $P$ contained in $S_i$.  

Furthermore, we will also suppose that $S$ is not doubly-ruled, i.e. that there are not two different families of rulings contained in $S$. It is well-known that the doubly-ruled surfaces are the plane, the hyperbolic paraboloid, and the single-sheeted hyperboloid (see \cite[\S I.3]{HC90}). For paraboloids and hyperboloids one can study affine equivalences by first computing the implicit equation, which is easy to do in the case of quadrics, and then applying matrix methods. 

In the rest of this paper we will assume that, for $i=1,2$, $\bfx_i(t,s)$ is \emph{proper}, i.e. that the parametrization in Eq. \eqref{surf} is injective except at most at a 1-dimensional subset of $S_i$; in particular, this implies that $\bfx_i^{-1}$ exists and is rational. Additionally, we need some more assumptions on Eq. \eqref{surf}. First, and this assumption will be important in order to develop our results, we will suppose that $\bfq_i(t)$ is polynomially parametrized, so that the components of $\bfq_i(t)$ have no denominators. We will also suppose that the (polynomial) components of $\bfq_i(t)$ are relatively prime, i.e. writing $\bfq_i(t)=(q_{1,i}(t),q_{2,i}(t),q_{3,i}(t))$, we will assume that $\gcd(q_{1,i}(t),q_{2,i}(t),q_{3,i}(t))=1$. 

Let us see that we can always achieve the two last requirements ($\bfq_i(t)$ polynomial, with components relatively prime), so that the above assumptions can be considered as completely general. Indeed, if some of these assumptions do not hold, then we can replace 
\begin{equation} \label{replace}
\bfq_i(t):=\mu_i(t) \bfq_i(t),\mbox{ }\mu_i(t)=\frac{\mu_{1,i}(t)}{\mu_{2,i}(t)},
\end{equation}
where $\mu_{1,i}(t)$ is the least common multiple of the denominators of the components of $\bfq_i(t)$, and $\mu_{2,i}(t)$ is the greatest common divisor of the numerators of the components of $\bfq_i(t)$. Notice that since $\mu_i(t) \bfq_i(t)$ is parallel to $\bfq_i(t)$ for all $t$, the new parametrization $\widehat{\bfx}_i(t,s)=\bfp_i(t)+s \mu_i(t)\bfq_i(t)$ also defines the surface $S_i$, because the rulings of the surfaces defined by $\bfx_i(t,s)$ and $\widehat{\bfx}_i(t,s)$ coincide. Furthermore, we can perform this substitution without losing properness, as shown by the next lemma. 

\begin{lemma}\label{prop-new}
Let $S$ be a ruled surface, and let $\bfx(t,s)$ be a proper parametrization of $S$ defined by Eq. \eqref{surf}. Let $\widehat{\bfx}(t,s)=\bfp(t)+s \mu(t)\bfq(t)$, where $\mu(t)$ is defined as in Eq. \eqref{replace}. Then $\widehat{\pmb{x}}(t,s)$ is also proper.
\end{lemma}

\begin{proof} Suppose that $\widehat{\pmb{x}}(t,s)$ is not proper. Then a generic point $P$ of $S$ is generated via $\widehat{\pmb{x}}(t,s)$ by two different pairs $(t_1,s_1)\neq(t_2,s_2)$; furthermore, since $P$ is generic we can assume that $\mu(t_1)\cdot \mu(t_2)\neq 0$. In this situation, $P$ is generated via $\bfx(t,s)$ by the pairs $(t_1,\tilde{s}_1)$, $(t_2,\tilde{s}_2)$ where $\tilde{s}_i=\mu(t_i)\cdot s_i$, $i=1,2$. Since $\bfx(t,s)$ is proper by hypothesis, we deduce that $t_1=t_2$, $\tilde{s}_1=\tilde{s}_2$, in which case $\mu(t_1)=\mu(t_2)$ too. And since $\tilde{s}_1=\tilde{s}_2$ and $\mu(t_i)\neq 0$, we conclude that $s_1=s_2$, contradicting that the pairs $(t_1,s_1)$ and $(t_2,s_2)$ are different.
\end{proof}

We say that $S_1,S_2$ are \emph{affinely equivalent} if there exists a nonsingular affine mapping $f:\RR^3\longrightarrow \RR^3$, 
\begin{equation}\label{eq:affine}
f({\bf x}) = \bfA {\bf x} + \bfb, \qquad {\bf x}\in \RR^3,
\end{equation}
with $\bfb \in \RR^3$ and $A\in \RR^{3\times 3}$ a nonsingular square matrix or order 3, such that $f(S_1)=S_2$. We say that $f$ is an \emph{affine equivalence} between $S_1,S_2$. If $\bfA$ is an orthogonal matrix, i.e. $\bfA^T \bfA=I$, where $I$ denotes the $3\times3$ identity matrix, we say that $f$ defines an \emph{isometry} between $S_1,S_2$, and that $S_1,S_2$ are \emph{isometric}. If $\bfA=\lambda \bfQ$ where $\bfQ$ is orthogonal and $\lambda\neq 0$, we say that $f$ defines a \emph{similarity} between $S_1,S_2$, and that $S_1,S_2$ are \emph{similar}. Furthermore, if $S_1=S_2$ and $f$ defines a non-trivial isometry of $S$ onto itself, we say that $f$ is a \emph{symmetry} of $S$, and that $S$ is \emph{symmetric}. 

Finally, we say that $f:\RR^n\longrightarrow \RR^n$ is an \emph{involution} if $f\circ f=\mbox{id}_{{\RR}^n}$. Involutions are particularly interesting when $S_1=S_2=S$ and we consider the symmetries of $S$, since notable symmetries like reflections on a plane, axial symmetries (i.e. symmetries with respect to a line, or equivalently rotations about a line of angle $\pi$) and central symmetries (i.e. symmetries with respect to a point) are involutions. Rotational symmetries, i.e. rotations leaving $S$ invariant, with rotation angle different from $\pi$, however, are not involutions.

\section{Affine equivalences of ruled surfaces.}\label{sec-symmetries}

Let $S_1,S_2$ be real rational ruled surfaces parametrized by $\bfx_1,\bfx_2$ as in Eq. \eqref{surf}, satisfying the conditions of Section \ref{gen-surf}, i.e. for $i=1,2$ we assume that: $\bfx_i(t,s)$ is proper; $\bfq_i(t)$ is polynomial with relatively prime components; $S_i$ is not doubly-ruled. Our goal in this section is to present a method to detect whether or not $S_1,S_2$ are affinely equivalent, and to compute the affine equivalences between $S_1,S_2$ in the affirmative case. In order to develop our method we will assume that $S_1,S_2$ are not cylindrical, so that $\pmb{q}_i(t)$, $i=1,2$, are not constant. Nevertheless, we will address cylindrical surfaces in Subsection \ref{subsec-cyl}. Notice that one can efficiently detect whether or not a rational surface is cylindrical by using the results in \cite{AG17}. 

The following result is crucial for us. 

\begin{theorem} \label{th-fundam}
 Let $S_1,S_2$ be two rational real ruled surfaces properly parametrized by $\bfx_1,\bfx_2$ as in Eq. \eqref{surf}. A mapping $f:{\RR}^3\to {\RR}^3$, $f({\bf x})=\bfA {\bf x}+\bfb$, with $\bfA\in {\Bbb R}^{3\times 3}$, $\bfb\in {\Bbb R}^3$ and $\bfA$ nonsingular, satisfies that $f(S_1)=S_2$, so that $S_1,S_2$ are affinely equivalent, if and only if there exists a birational transformation $\varphi:{\RR}^2\to {\RR}^2$, such that the diagram 
	\begin{equation}\label{eq:fundamentaldiagram}
	 \xymatrix{
		S_1 \ar[r]^{f} & S_2 \\
		\RR^2 \ar@{-->}[u]^{\pmb{x}_1} \ar@{-->}[r]_{\varphi} & \RR^2 \ar@{-->}[u]_{\pmb{x}_2}
		}
	\end{equation}
	is commutative. In particular, for a generic point $(t,s)\in \RR^2$
	\begin{equation}\label{fundam-eq}  f\circ \bfx_1=\bfx_2\circ \varphi.
	\end{equation}
\end{theorem}
																																									
\begin{proof}
``$\Rightarrow$" Since $\bfx_2$ is proper by hypothesis, $\bfx_2^{-1}$ exists and is rational. Therefore, $\varphi=\bfx_2^{-1}\circ f \circ \bfx_1$ is birational, because $\varphi$ is the composition of birational transformations. ``$\Leftarrow$" Since $f\circ \pmb{x}_1=\pmb{x}_2\circ \varphi$, whenever $\bfx_1(t,s)$ and $(\bfx_2\circ \varphi)(t,s)$ are well-defined $(f\circ \bfx_1)(t,s)\in S_2$, so $f(S_1)\subset S_2$. Since $f$ is nonsingular, $f(S_1)$ defines a rational surface, i.e. $f(S_1)$ does not degenerate into a curve. Additionally both $f(S_1),S_2$ are rational, and therefore irreducible; since $f(S_1)\subset S_2$ and $f(S_1),S_2$ are irreducible, $f(S_1)=S_2$, i.e. $S_1,S_2$ are affinely equivalent.
\end{proof}

Additionally, from Eq. \eqref{fundam-eq} one can easily see that each affine mapping $f$ is associated with a different $\varphi$. 

From Theorem \ref{th-fundam} we observe that $\varphi$ is a birational transformation of the plane. Such a transformation is called a \emph{Cremona transformation}. However, unlike the birational transformations of the line, which are the well-known 
\emph{M\"obius transformations}, i.e. the transformations of the type
\begin{equation}\label{eq:Moebius}
\psi: {\RR} \dashrightarrow {\RR}, \qquad \psi(t) = \frac{\alpha t + \beta}{\gamma t + \delta}, \qquad \alpha \delta - \beta \gamma \neq 0,
\end{equation}
Cremona transformations do not have a generic closed form. Therefore, in order to describe $\varphi$, we need to make use of the properties of the surfaces we are investigating, in this case of the fact that they are ruled. The following result provides a first clue in this direction.

\begin{proposition}\label{fund}
Let $S_1,S_2$ be rational ruled surfaces properly parametrized as in Eq. \eqref{surf}, which are not doubly ruled. Let $f({\bf x})=\bfA {\bf x}+\bfb$ be a nonsingular affine mapping satisfying $f(S_1)=S_2$, and let $\varphi:\RR^2\to \RR^2$ be the birational transformation making the diagram in Eq. \eqref{eq:fundamentaldiagram} commutative. Then
	\begin{equation}\label{phi-funct}
	\varphi(t,s)=(\psi(t),a(t)\cdotp s+ c(t)),
	\end{equation}
	where $\psi(t)$ is a M\"obius transformation and $a(t),c(t)$ are rational functions.
\end{proposition}
 
\begin{proof} Since $f$ is an affine mapping, $f$ maps rulings of $S_1$ onto rulings of $S_2$. Let $\varphi(t,s)=(\varphi_1(t,s),\varphi_2(t,s))$. A generic ruling of $S_i$, with $i=1,2$ is defined by 
$\pmb{x}_i(t_{a_i},s)$, where $t_{a_i}$ is a constant. Since $S_2$ is not doubly ruled, the ruling parametrized by $\pmb{x}_1(t_{a_1},s)$ is mapped by $f$ onto the ruling parametrized by $\pmb{x}_2(t_{a_2},s)$. Using Eq. \eqref{fundam-eq}, we get 
$$f(\bfx_1(t_{a_1},s))=\bfx_2(\varphi(t_{a_1},s))=\bfx_2(\varphi_1(t_{a_1},s),\varphi_2(t_{a_1},s)),$$so $\varphi_1(t_{a_1},s)=t_{a_2}$, i.e. $\varphi_1(t_{a_1},s)$ does not depend on $s$. Since this independence happens for a generic $t_{a_1}$, we deduce that $\varphi_1(t,s)=\varphi_1(t)$. Since $\varphi$ is birational, $\varphi_1$ is birational as well; in particular, we deduce that $\varphi_1$ is a birational transformation of the line, so $\varphi_1$ must be a M\"obius transformation, which we represent by $\psi(t)$. The rest of the theorem follows from Eq. \eqref{fundam-eq}, taking into account that $f({\bf x})=\bfA {\bf x}+\bfb$. 
\end{proof}

Let us now investigate the structure of the function $a(t)$ in Eq. \eqref{phi-funct}. Recall that $\pmb{x}_i(t,s)=\bfp_i(t)+s\cdotp\pmb{q}_i(t)$, where $\bfq_i(t)=(q_{i,1}(t),q_{i,2}(t),q_{i,3}(t))$, each $q_{i,j}(t)$ is polynomial and $\gcd(q_{i,1},q_{i,2},q_{i,3})=1$. Also, let 

\begin{equation}\label{eq-n}
n_i=\mbox{max}\{\deg(q_{i,1}(t)),\deg(q_{i,2}(t)),\deg(q_{i,3}(t))\},
\end{equation}

\noindent and let us write \[a(t)=\dfrac{A(t)}{B(t)},\mbox{ }\psi(t) = \dfrac{\alpha t + \beta}{\gamma t + \delta},\]where $A,B\in\RR[t]$, $\gcd(A,B)=1$, and $\alpha \delta - \beta \gamma \neq 0$.  Combining Eq. \eqref{phi-funct} and Eq. \eqref{fundam-eq} with $f({\bf x})=\bfA {\bf x}+\bfb$, and comparing the coefficients of $s$, we get
\begin{equation}\label{at}
		\bfA\cdotp\bfq_1(t)=a(t)\cdotp\bfq_2(\psi(t)).
	\end{equation}
Since $\bfq_i(t)$, $i=1,2$, is polynomial, the left hand-side of Eq. \eqref{at} is polynomial as well, so the right hand-side of Eq. \eqref{at} must also be polynomial. This observation yields the following results; here, we denote the entries of the matrix $\bfA$ by $\bfA_{ij}$. 

\begin{lemma}\label{At}
	$(\gamma t + \delta)^{n_2}$ divides $A(t)$.
\end{lemma}

\begin{proof}
	From Eq. \eqref{at}, for $i=1,2,3$ we get 
	\begin{equation}\label{est}
	\bfA_{i1}\cdotp q_{1,1}(t)+\bfA_{i2}\cdotp q_{1,2}(t)+\bfA_{i3}\cdotp q_{1,3}(t)=a(t)\cdotp q_{2,i}(\psi(t)),
	\end{equation} where $q_{2,i}(t)=a_{\ell_i} t^{\ell_i}+a_{\ell_i-1}t^{\ell_i-1}+\cdots+a_0,$ with $\ell_i\leq n_2$ for $i\in \{1,2,3\}$. Furthermore, $\ell_i=n_2$ for at least one $i\in \{1,2,3\}$. Additionally, 
\begin{equation}\label{q2i}	
	q_{2,i}(\psi(t))=\dfrac{a_{\ell_i}(\alpha t+\beta)^{\ell_i}+a_{\ell_i-1}(\alpha t+\beta)^{\ell_i-1}(\gamma t+\delta)+\cdots+a_0(\gamma t+\delta)^{\ell_i}}{(\gamma t+\delta)^{\ell_i}}.
\end{equation}	
	Since $\gamma t +\delta$ does not divide $\alpha t+\beta$, the numerator and denominator of $q_{2,i}(\psi(t))$ are relatively prime. Since the left hand-side of Eq. \eqref{est} is a polynomial, $a(t)\cdotp q_{2,i}(\psi(t))$ must be a polynomial as well, so $(\gamma t+\beta)^{\ell_i}$ divides $A(t)$. Since $\ell_i=n_2$ for some $i\in \{1,2,3\}$, the statement follows. 
\end{proof}

\begin{lemma}\label{Bt}
	$B(t)$ is a constant.
\end{lemma}

\begin{proof}
Let $N_i(t)$ be the numerator of $q_{2,i}(\psi(t))$, and recall that $\gcd(q_{2,1},q_{2,2},q_{2,3})=1$. Since the left hand-side of Eq. \eqref{est} is a polynomial, $B(t)|N_i(t)$ for $i=1,2,3$. Thus, $B(t)|G(t)$, where $G=\gcd(N_1,N_2,N_3)$. Now suppose that $G(t)$ is not constant. Then $N_1,N_2,N_3$ have a common root $t_0$. Moreover, since the numerators and denominators of the $q_{2,i}(\psi(t))$ are relative prime, $\gamma t_0+\delta\neq 0$. Therefore, $\psi(t_0)$ is well defined and $\psi(t_0)$ is a common root of the $q_{2,i}(t)$, because $q_{2,i}(\psi(t_0))=\frac{N_i(t_0)}{(\gamma t_0+\delta)^{\ell_i}}$. But this contradicts the fact that $\gcd(q_{2,1},q_{2,2},q_{2,3})=1$. Thus, $G(t)$ is constant and since $B(t)|G(t)$, $B(t)$ must be a constant.
\end{proof}

Finally, we get the following proposition on the form of the function $a(t)$.

\begin{proposition} \label{prop-a(t)}
The function $a(t)$ satisfies that $a(t)=k\cdotp(\gamma t+\delta)^{n_2}$, where $k$ is a nonzero constant.
\end{proposition}

\begin{proof}From the two previous lemmas we have $a(t)=k(t)\cdotp(\gamma t+\delta)^{n_2}$ for some polynomial $k(t)$. Additionally, from Eq. \eqref{at}
\begin{equation}\label{eq}
	\bfA\cdotp\bfq_1(t)=k(t)\cdotp(\gamma t+\delta)^{n_2}\cdotp\bfq_2(\psi(t)).
	\end{equation}
Taking Eq. \eqref{q2i} into account, we observe that $(\gamma t+\delta)^{n_2}\cdotp\bfq_2(\psi(t))$ is polynomial. If $k(t)$ is not a constant, then the components of $\bfA\cdot \bfq_1(t)$ are not relatively prime, i.e. $\bfA\cdot \bfq_1(t)=r(t)\widehat{\bfq}_1(t)$, with $r(t)$ nonconstant, and $\widehat{\bfq}_1(t)$ a polynomial parametrization with relatively prime components. However, since $\bfA$ is nonsingular, in that case we have $\bfq_1(t)=r(t)\bfA^{-1}\widehat{\bfq}_1(t)$, which implies that the components of $\bfq_1(t)$ are not relatively prime either. Since by hypothesis the components of $\bfq_1(t)$ are relatively prime, $k(t)$ must be a constant $k$. Finally, since $\bfA$ is nonsingular, from Eq. \eqref{eq} we get that $k\neq 0$.  
	\end{proof}

Taking Proposition \ref{prop-a(t)} into account and Eq. \eqref{eq}, we get the following corollary.
	
\begin{corollary}\label{cor-eq-deg}
If $S_1,S_2$ are affinely equivalent, then $n_1=n_2$. 
\end{corollary}
	
We summarize the previous results in the following theorem. In the rest of the paper, we denote, according to Corollary \ref{cor-eq-deg}, $n_1=n_2=n$.

\begin{theorem}\label{fund2}
Let $S_1,S_2$ be two rational ruled surfaces properly parametrized as in Eq. \eqref{surf}, which are not doubly ruled. Let $\bfq_i(t)=(q_{i,1}(t),q_{i,2}(t),q_{i,3}(t))$, with $q_{i,j}(t)\in {\RR}[t]$ for $i=1,2$ and $j=1,2,3$, where $n_1=n_2=n$. Let $f({\bf x})=\bfA {\bf x}+\bfb$, with $\bfA$ nonsingular, such that $f(S_1)=S_2$, and let $\varphi:\RR^2\to \RR^2$ be the birational transformation making the diagram in Eq. \eqref{eq:fundamentaldiagram} commutative. Then
	\begin{equation}\label{phi-funct-2}
	\varphi(t,s)=(\psi(t),k\cdotp(\gamma t+\delta)^{n}\cdotp s+ c(t)),
	\end{equation}
	where $\psi(t)$ is a M\"obius transformation, $k$ is a constant, and $c(t)$ is a rational function, satisfying that 
\begin{equation}\label{eqmain}
\bfA\cdotp\bfq_1(t)=k\cdotp(\gamma t+\delta)^{n}\cdotp\bfq_2(\psi(t)).	
\end{equation}	
\end{theorem}

The equation Eq. \eqref{eqmain} can be interpreted in geometric terms. In order to do this, it is clearer to write Eq. \eqref{eqmain} projectively. Let $\tilde{\bfq}_i(t,\omega)=[q_{i,1}(t,\omega):q_{i,2}(t,\omega):q_{i,3}(t,\omega)]\in {\Bbb P}^2({\Bbb R})$, where $i=1,2$ and $\omega$ is a homogenization variable. Then Eq. \eqref{eqmain} can be written as 
\begin{equation}\label{proj-main}
\bfA\cdot \tilde{\bfq}_1(t,\omega)=k\cdot \tilde{\bfq}_2(\alpha t+\beta \omega,\gamma t+\delta \omega).
\end{equation}
What Eq. \eqref{proj-main} is expressing (see Section 3 of \cite{HJ18}) is the fact that the projective curves defined by $\tilde{\bfq}_1(t,\omega)$ and $\tilde{\bfq}_2(t,\omega)$ are \emph{projectively equivalent}, and even more, that $\bfA$ defines a projectivity mapping the projective curve defined by $\tilde{\bfq}_1(t,\omega)$ onto the projective curve defined by $\tilde{\bfq}_2(t,\omega)$ (or $k\cdot \tilde{\bfq}_2(t,\omega)$, since projectively $\tilde{\bfq}_2(t,\omega)$ and $k\cdot \tilde{\bfq}_2(t,\omega)$ can be identified). This makes perfect sense from a geometric point of view: affine equivalences map rulings of $S_1$ onto rulings of $S_2$, as observed in the proof of Proposition \ref{fund}, and $\tilde{\bfq}_1(t,\omega)$, $\tilde{\bfq}_2(t,\omega)$ define the \emph{directions} of these rulings. The matrix $\bfA$ defines the map sending the direction of each ruling of $S_1$, onto the direction of another ruling of $S_2$. 

Projective equivalences between curves in any dimension, and in particular systems of equations like Eq. \eqref{proj-main} (and therefore Eq. \eqref{eqmain}) are studied in great detail in \cite{HJ18}. We will benefit from the study carried out in \cite{HJ18} in the next section, where we address the computation of the affine equivalences between $S_1,S_2$. 

Finally, from Eq. \eqref{fundam-eq} and Eq. \eqref{phi-funct-2}, we get the relationship 
\begin{equation}\label{para-c}
\bfA\bfp_1(t)+\bfb=\bfp_2(\psi(t))+c(t)\bfq_2(\psi(t)).
\end{equation}

We will see how to exploit Eq. \eqref{eqmain} and Eq. \eqref{para-c} in the coming section. 

\section{Computation of the affine equivalences.} \label{sec-computation}

The computation of the affine equivalences between $S_1,S_2$ is based on the following result, which in turn follows from the results of the previous section.

\begin{proposition}\label{prop-comp}
The affine equivalences $f({\bf x})=\bfA{\bf x}+\bfb$ between $S_1,S_2$ correspond to the $\bfA\in{\Bbb R}^{3\times 3}$, $\bfb\in {\Bbb R}^3$ satisfying \emph{Eq.} \eqref{eqmain} and \emph{Eq.} \eqref{para-c}, where $\emph{\mbox{det}}(\bfA)\neq 0$, $k\neq 0$, $\psi(t)=\frac{\alpha t+\beta}{\gamma t+\delta}$ and $\alpha\delta-\beta\gamma \neq 0$. 
\end{proposition}

Notice that since the components of $\bfq_2(t)$ are polynomials of degree at most $n$, Eq. \eqref{eqmain} only involves polynomials, and provides equations which are linear in the entries $\bfA_{ij}$ of the matrix $\bfA$; furthermore, the coefficients of the $\bfA_{ij}$ in these linear equations are constants, while the constant terms of these linear equations depend on $\alpha,\beta,\gamma,\delta$ and $k$. However, Eq. \eqref{para-c} involves rational functions, i.e. polynomial denominators. Additionally, since $\alpha\delta-\beta\gamma\neq 0$, we can always assume either that $\alpha\delta-\beta\gamma=1$, or separate the analysis in two different cases, namely the case $\gamma=1$, and the case $\gamma=0$, $\delta=1$, which allows to perform the computation with fewer variables (although twice). 

The computation proceeds in three different steps, (A), (B), (C). Let us describe these steps in detail. 

\vspace{2 mm}

\noindent \emph{(A) Writing $\bfA$ in terms of $\alpha,\beta,\gamma,\delta$, and $k$.} At this step we exploit Eq. \eqref{eqmain}, which has been studied in great detail in Section 3 of \cite{HJ18}. Writing Eq. \eqref{eqmain} in components, we get 

\begin{equation}\label{incompo}
\left\{\begin{array}{ccc}
\bfA_{11}\cdotp q_{1,1}(t)+\bfA_{12}\cdotp q_{1,2}(t)+\bfA_{13}\cdotp q_{1,3}(t) & = & k(\gamma t+\beta)^n q_{2,1}(\psi(t)),\\
\bfA_{21}\cdotp q_{1,1}(t)+\bfA_{22}\cdotp q_{1,2}(t)+\bfA_{23}\cdotp q_{1,3}(t) & = & k(\gamma t+\beta)^n q_{2,2}(\psi(t)),\\
\bfA_{31}\cdotp q_{1,1}(t)+\bfA_{32}\cdotp q_{1,2}(t)+\bfA_{33}\cdotp q_{1,3}(t) & = & k(\gamma t+\beta)^n q_{2,3}(\psi(t)).
\end{array}\right.
\end{equation}

Since the $q_{2,j}(t)$ have degree at most $n$, the expressions at the right hand-side of Eq. \eqref{incompo} are, in fact, polynomials. Setting equal the coefficients of $t^\ell$, where $\ell=0,1,\ldots,n$, at both sides od Eq. \eqref{incompo}, we derive a system ${\mathcal L}$, linear in the $\bfA_{ij}$, where the coefficients of the $\bfA_{ij}$ are constant numbers, and where the constant terms are polynomials in $\alpha,\beta,\gamma,\delta$ and $k$. Let us write $\bfq_1(t)$ as
\begin{equation}\label{qcomponents}
\bfq_1(t)={\bf v}_0+{\bf v}_1 t+\cdots+{\bf v}_n t^n,
\end{equation}
where ${\bf v}_\ell\in {\Bbb R}^3$, for $\ell=0,\ldots,n$, is a numeric row vector whose components are the coefficients in $t^\ell$ of $q_{1,1}(t),q_{1,2}(t)$ and $q_{1,3}(t)$, respectively. Then the system ${\mathcal L}$ has the form: 

\begin{equation}\label{sysL}
\underbrace{\left[\begin{array}{rcl}
\begin{array}{c}
{\bf v}_0  \\ 
\vdots     \\
{\bf v}_n 
\end{array} & & \\ 
& \begin{array}{c}
{\bf v}_0  \\ 
\vdots     \\
{\bf v}_n 
\end{array} & \\
& & \begin{array}{c}
{\bf v}_0  \\ 
\vdots     \\
{\bf v}_n 
\end{array} \\
\end{array}\right]}_{\mathcal A}
\cdot \begin{bmatrix} 
\bfA_{11}\\
\bfA_{12}\\
\bfA_{13}\\
\bfA_{21}\\
\vdots \\
\bfA_{33} 
\end{bmatrix} = \left[\begin{array}{l}\bullet_1 \\ \vdots \\ \bullet_{n+1} \\ \vdots \\ \bullet_{2(n+1)} \\ \vdots \\ \bullet_{3(n+1)} \end{array}\right]
\end{equation}

Here we see that ${\mathcal A}\in {\Bbb R}^{3(n+1)\times 9}$ is a block matrix with three nonzero blocks of size $(n+1)\times 3$, consisting of the row vectors ${\bf v}_0,\ldots,{\bf v}_n$. The constant terms $\bullet_j$, where $j=1,\ldots,3(n+1)$, are products of $k$ times a homogeneous polynomial in $\alpha,\beta,\gamma,\delta$ of degree $n$, a structure observed in Section 3.2 of \cite{HJ18}. Notice also that the number of $3(n+1)$ equations is in agreement with the observations raised in Section 3 of \cite{HJ18} (compare to Table 2 in Section 3 of \cite{HJ18}, taking into account that we are dealing with projective curves, defined by $\tilde{\bfq}_1,\tilde{\bfq}_2$, in the projective plane).

Let $r=\mbox{rank}({\bf v}_0,\ldots,{\bf v}_n)$; notice that since ${\bf v}_{\ell}\in {\Bbb R}^3$, we get $r\leq 3$. Furthermore, if $r=2$ then the directions of all the rulings of $S_1$ are parallel to a plane, and if $r=1$ then all the rulings of $S_1$ are parallel to a same vector ${\bf v}$, i.e. $S_1$ is a \emph{cylindrical} surface; this special case is much easier to solve, see Subsection \ref{subsec-cyl}. 

Now by the structure of the matrix ${\mathcal A}$ we get $\mbox{rank}({\mathcal A})=3r$. Let us address the cases $r=3$ and $r=2$. The case $r=3$ is analyzed in detail in Section 3.2 of \cite{HJ18}; here we adapt several results of \cite{HJ18} to our case. However, case $r=2$ is, apparently, not addressed in \cite{HJ18}.

\begin{itemize}
\item [(1)] {\bf Case} $r=3$: since $\mbox{rank}({\mathcal A})=3r$, for $r=3$ we get $\mbox{rank}({\mathcal A})=9$, so we can solve the system ${\mathcal L}$ and write the $\bfA_{ij}$ in terms of $\alpha,\beta,\gamma,\delta$ and $k$. Additionally, applying Gauss-Jordan method to the system ${\mathcal L}$ we get $3(n+1)-3r$ additional conditions on $\alpha,\beta,\gamma,\delta$ and $k$ that must hold for ${\mathcal L}$ to be consistent; when $r=3$, we get $3n-6$ conditions of this type. These conditions are products of $k$ times a homogeneous polynomial in $\alpha,\beta,\gamma,\delta$. Since $k\neq 0$, we can factor out $k$ and get $3n-6$ homogeneous conditions on $\alpha,\beta,\gamma,\delta$ alone, of degree $n$. Since $\alpha\delta-\beta\gamma\neq 0$, one can add the extra condition $\alpha\delta-\beta\gamma=1$. 

This way we get a polynomial system $P_{\mathcal A}$ in $\alpha,\beta,\gamma,\delta$: if this polynomial system is not consistent, the surfaces $S_1,S_2$ are identified as non-affinely equivalent, and the computation stops. Otherwise we can get either tentative values for $\alpha,\beta,\gamma,\delta$ that may or may not give rise to an affine equivalence between $S_1,S_2$ (this must be tested later), or a number of relations between the $\alpha,\beta,\gamma,\delta$. If these relations allow to write some of these parameters in terms of the others, we can reduce the number of parameters in the subsequent computations. 

Notice that when $n=2$, we get $3n-6=3\cdot 2-6=0$, so no extra conditions in $\alpha,\beta,\gamma,\delta$ are obtained. However, we can still write the ${\bfA}_{ij}$ in terms of $\alpha,\beta,\gamma,\delta$ and $k$.

\item [(2)] {\bf Case} $r=2$: in this case, since $r=2$ applying Gauss-Jordan method to the system ${\mathcal L}$ we get $3(n+1)-3\cdot 2=3n-3$ additional conditions on $\alpha,\beta,\gamma,\delta,k$ that must hold for ${\mathcal L}$ to be consistent, with the same properties as in the case before; also as before, we denote the collection of all these polynomial conditions by $P_{\mathcal A}$. However, since $\mbox{rank}({\mathcal A})=6$ is less than the number of $\bfA_{ij}$, in this case we cannot write all the $\bfA_{ij}$ in terms of $\alpha,\beta,\gamma,\delta,k$ only, i.e. three of the $\bfA_{ij}$ must act as parameters as well. This makes sense from a geometric point of view: if $r=2$ then $\bfq_1(t),\bfq_2(t)$ parametrize projective lines, and there are infinitely many projective transformations mapping a projective line onto another projective line. 
\end{itemize}

Observe that when the components of $\bfq_1(t)$ are linear, we are always either in the case $r=1$, or in the case $r=2$. In this last case, since $n=1$ the number $3n-3$ of extra conditions vanishes, so we get no extra conditions on $\alpha,\beta,\gamma,\delta$. 

Summarizing, at this step we write either all the $\bfA_{ij}$, when $r=3$, or only six of the $\bfA_{ij}$, when $r=2$, in terms of $\alpha,\beta,\gamma,\delta,k$. Furthermore, except in the case $r=3,n=2$ and the case $n=1$, we get polynomial conditions on $\alpha,\beta,\gamma,\delta$, which may help to either detect that the surfaces are not affine equivalent (when these conditions are not compatible), or to reduce the number of parameters.  

\vspace{2 mm}

\noindent \emph{(B) Writing $\bfb$ in terms of $\alpha,\beta,\gamma,\delta$, and $k$, and computing $c(t)$.} Writing Eq. \eqref{para-c} in components, we get

\begin{equation}\label{system}
\left\{\begin{array}{ccc}
\bfA_{11}\cdotp p_{1,1}(t)+\bfA_{12}\cdotp p_{1,2}(t)+\bfA_{13}\cdotp p_{1,3}(t)+b_1 & = & p_{2,1}(\psi(t))+c(t)q_{2,1}(\psi(t)),\\
\bfA_{21}\cdotp p_{1,1}(t)+\bfA_{22}\cdotp p_{1,2}(t)+\bfA_{23}\cdotp p_{1,3}(t)+b_2 & = & p_{2,2}(\psi(t))+c(t)q_{2,2}(\psi(t)),\\
\bfA_{31}\cdotp p_{1,1}(t)+\bfA_{32}\cdotp p_{1,2}(t)+\bfA_{33}\cdotp p_{1,3}(t)+b_3 & = & p_{2,3}(\psi(t))+c(t)q_{2,3}(\psi(t)),
\end{array}\right.
\end{equation}

\noindent where we assume that the $\bfA_{ij}$, or some of the $\bfA_{ij}$, have already been written in terms of $\alpha,\beta,\gamma,\delta,k$. Now we proceed as follows: 

\begin{itemize}
\item [(i)] Eliminating $c(t)$ between the first and second equations of Eq. \eqref{system} provides an equation $E_1$ linear in $b_1,b_2$, with coefficients that are rational functions of $t$. 
\item [(ii)] Proceeding in the same way with the second and third equations, we get an equation $E_2$, linear in $b_2,b_3$. 
\item [(iii)] We evaluate $E_1$ and $E_2$ at several random $t$-values. This way we get a linear system in $b_1,b_2,b_3$, whose solution provides $\bfb$. 
\item [(iv)] Finally, we compute $c(t)$ from any equation of Eq. \eqref{system}. 
\end{itemize}

We will refer later to this procedure as ``the steps (i)-(iv)". 

\vspace{2 mm}

\noindent \emph{(C) Deriving a polynomial system ${\mathcal S}$, and computing the affine equivalences.} Substituting the expressions for $\bfA,\bfb$ and $c(t)$ computed in the steps (A) and (B) into Eq. \eqref{fundam-eq}, we obtain a polynomial system ${\mathcal S}$. If $r=3$, the unknowns of ${\mathcal S}$ are, at most, $k,\alpha,\beta,\gamma,\delta$, and we can have fewer unknowns if the polynomial conditions $P_{\mathcal A}$ in step (A) allow to write some of these variables in terms of the others. If $r=2$, we can have at most three more unknowns besides $k,\alpha,\beta,\gamma,\delta$, namely three of the $\bfA_{ij}$; again, the polynomial system $P_{\mathcal A}$ may help reduce the total number of parameters, and therefor of unknowns in ${\mathcal S}$. Thus, the number of unknowns in ${\mathcal S}$ is $\leq 5$, if $r=3$, and $\leq 8$, if $r=2$. 

The solutions of this polynomial system provide the affine equivalences between $S_1,S_2$. We summarize the whole procedure to find the affine equivalences between $S_1,S_2$ in Algorithm {\tt Affine-Eq-Ruled}.

\begin{algorithm}[t!]
\begin{algorithmic}[1]
\REQUIRE Two ruled surfaces $S_1,S_2$, properly parametrized by $\bfx_i(t,s)=\bfp_i(t)+s\bfq_i(t)$, $i=1,2$, where each $\bfq_i(t)$ is polynomial and with relatively prime components of degree $\leq n$.
\ENSURE The affine equivalences $f({\bf x})=\bfA {\bf x}+\bfb$ between $S_1,S_2$.
\STATE{Compute the system ${\mathcal L}$ in Eq. \eqref{sysL}.} 
\STATE{Apply the Gauss-Jordan method on the system ${\mathcal L}$.}
\IF{$r=3$ and $n\geq 3$, or $r=2$ and $n\geq 2$}
\STATE{solve the polynomial system $P_{\mathcal A}$ in $\alpha,\beta,\gamma,\delta$.}
\IF{$P_{\mathcal A}$ is not consistent}
\STATE{{\bf return} {\tt $S_1$ and $S_2$ are not affinely equivalent}, and {\bf stop}}
\ENDIF 
\ENDIF
\STATE{Solve the system ${\mathcal L}$}
\STATE{Write the solutions of ${\mathcal L}$ with as few variables as possible, using, if any, the solutions of $P_{\mathcal A}$}
\STATE{Follow the steps (i)-(iv) to write $\bfb$ in terms of the variables in the step before, and to compute $c(t)$}
\STATE{Substitute $\bfA,\bfb,c(t)$ and the $\varphi$ in Eq. \eqref{phi-funct-2} into Eq. \eqref{fundam-eq}}
\STATE{Derive from the preceding substitution a polynomial system ${\mathcal S}$ in the parameters appearing in Step 9}
\IF{no solution is found}
\STATE{{\bf return} {\tt $S_1$ and $S_2$ are not affinely equivalent.}}
\ELSE
\FOR{each solution found}
\STATE{compute the corresponding mapping $f({\bf x})=\bfA {\bf x}+\bfb$}
\ENDFOR
\ENDIF
\end{algorithmic}
\caption{Affine-Eq-Ruled}\label{equiv-alg}
\end{algorithm}

\subsection{The special case of conical surfaces.} \label{subsec-conical}

We say that $S$ is a \emph{conical} surface if all the rulings of $S$ intersect at one point ${\bf p}_0\in S$, called the \emph{vertex}. The vertex can be computed by using the results in \cite{AG17}, and by applying a translation if necessary, we can always assume that ${\bf p}_0$ is the origin. Therefore, if $S$ is rational and properly parametrized we can assume that $S$ is given by means of a parametrization $\bfx(t,s)=s\bfq(t)$, where $\bfq(t)$ is polynomial. 

Now given two rational conical surfaces $S_1,S_2$ parametrized by $\bfx_i(t,s)=s\bfq_i(t)$, with $\bfq_i(t)$ polynomial for $i=1,2$, any affine equivalence between $S_1,S_2$ has the form $f({\bf x})=\bfA {\bf x}$, so $\bfb={\bf 0}$. Since $p_1(t),p_2(t)$ are identically zero, we get that the function $c(t)$ is identically zero as well, and therefore Eq. \eqref{para-c} is reduced to $0=0$. Thus, the computation of the affine equivalences between $S_1,S_2$ reduces to solving Eq. \eqref{eqmain}. Notice as well that the system derived from Eq. \eqref{eqmain} is homogeneous in $k$ and the entries of the matrix $\bfA$, which implies that $\bfA$ is defined only up to a multiplicative constant. This makes perfect sense, since any conical surface is invariant by homotheties where the homothety center is the vertex.

\subsection{The special case of cylindrical surfaces.} \label{subsec-cyl}

Under the assumption that $\bfq_1(t),\bfq_2(t)$ are polynomial, and with relative prime components, $S_1,S_2$ are cylindrical iff the $\bfq_i(t)$ are constant vectors. These vectors define the direction of all the rulings of $S_1,S_2$. Then in order to check whether or not $S_1,S_2$ are affinely equivalent, it suffices to check whether or not the planar curves ${\mathcal C}_1,{\mathcal C}_2$, obtained by intersecting $S_1,S_2$ with planes $\Pi_1,\Pi_2$ respectively normal to the direction of $\bfq_1(t),\bfq_2(t)$, are affinely equivalent. This can be done, for instance, by using the algorithm in \cite{HJ18}. Notice that the affine equivalences of $S_1,S_2$ are, in this case, the affine equivalences of the plane sections followed by any translations along the direction of the rulings of $S_2$, and any dilatation in the same direction.

\subsection{Computing isometries and symmetries.} \label{spec-sub-sym}

Let us address now the case when the affine mapping $f({\bf x})=\bfA {\bf x}+\bfb$ is orthogonal, in which case $f$ is an isometry between $S_1,S_2$. In order to find the isometries between $S_1,S_2$ we can certainly apply Algorithm {\tt Affine-Eq-Ruled}, with the extra condition that $\bfA$ is orthogonal. However, in this case we have additional conditions, which may be an advantage in order to simplify the computation. Indeed, since orthogonal mappings preserve norms, taking norms in Eq. \eqref{at}, with $k$ a constant, we reach the condition

\begin{equation}\label{eq-compu}
\Vert \bfq_1(t)\Vert^2-k^2\cdot (\gamma t+\delta)^{2n}\cdot \Vert \bfq_2(\psi(t))\Vert^2=0.
\end{equation}

\noindent Setting all the coefficients in $t$ at the left hand-side of Eq. \eqref{eq-compu} equal to zero, we get a polynomial system ${\mathcal P}$ of $2n+1$ equations, each one consisting of a homogeneous polynomial of degree $2n$ in the variables $\alpha,\beta,\gamma,\delta$ multiplied by $k^2$, plus a constant. These equations have a higher degree than the equations of the polynomial system $P_{\mathcal A}$, which were of degree $n$. However, collecting the equations in $P_{\mathcal A}$ and ${\mathcal P}$ provides a bigger polynomial system in $\alpha,\beta,\gamma,\delta,k$, which may help to reduce the total number of parameters in the polynomial system ${\mathcal S}$, and$/$or the number of tentative values for $\alpha,\beta,\gamma,\delta,k$. In particular, in the cases $r=3,n=2$ and $n=1$ applying Algorithm {\tt Affine-Eq-Ruled} does not provide extra conditions on $\alpha,\beta,\gamma,\delta,k$; however, Eq. \eqref{eq-compu} does.  

If $S_1=S_2=S$, the isometries leaving $S$ invariant are the symmetries of $S$. We can find the symmetries of $S$ by proceeding as before with $S_1=S_2$. However, recall from Section \ref{gen-surf} that certain notable symmetries, like central symmetries, axial symmetries and reflections on a plane, are involutions, i.e. affine mappings $f$ satisfying $f\circ f=\mbox{id}_{{\RR}^3}$. If we are interested only in involutions, we can improve the computation in the following way, First, from Eq. \eqref{fundam-eq}, one can see that $f\circ f=\mbox{id}_{{\RR}^3}$ iff the corresponding $\varphi$ satisfies $\varphi\circ \varphi=\mbox{id}_{{\RR}^2}$. Since we know from Theorem \ref{fund2} that 
\[
\varphi(t,s)=(\varphi_1(t,s),\varphi_2(t,s))=(\psi(t),s\cdot k(\gamma t+\delta)^n+c(t)),
\]
imposing here that $(\varphi\circ \varphi)(t,s)=(t,s)$ one gets two constraints:

\begin{itemize}
\item [(i)] $(\varphi_1\circ \varphi_1)(t,s)=t$, i.e. $(\psi\circ \psi)(t)=t$. In turn, this implies that 
\[
\alpha^2-\delta^2=0,\mbox{ }\beta(\alpha+\delta)=0,\mbox{ }\gamma(\alpha+\delta)=0.
\]
Therefore, either $\alpha=-\delta$, or $\alpha+\delta\neq 0$ and $\alpha=\delta$, $\beta=\gamma=0$. 
\item [(ii)] $\varphi_2(\varphi_1(t),\varphi_2(t,s))=s$, which implies 
\[
\left[s\cdot k(\gamma t+\delta)^n+c(t)\right]\cdot k\cdot \left[\gamma\cdot \frac{\alpha t +\beta}{\gamma t+\delta}+\delta\right]^n+c(\psi(t))=s.
\]
Comparing coefficients of $s$, we deduce that
\[
k^2\cdot \left[\gamma(\alpha +\delta)t+(\gamma \beta+\delta^2)\right]^n=1,
\]
which in turn yields 
\[
\gamma(\alpha+\delta)=0,\mbox{ }k^2(\gamma \beta+\delta^2)^n=1.
\]
Thus, either $\alpha=-\delta$ and $k^2(\gamma \beta+\delta^2)^n=1$, or $\alpha=\delta$, $\gamma=0$ and $k^2\delta^{2n}=1$.
\end{itemize}

Putting (i) and (ii) together, we get the following result, which allows to drop the total number of parameters, and therefore of unknowns in the polynomial system ${\mathcal S}$. 

\begin{theorem}\label{fund3}
Let $S$ be a rational ruled surface properly parametrized as in Eq. \eqref{surf}, which is not doubly ruled. Let $\bfq(t)=(q_{1}(t),q_{2}(t),q_{3}(t))$, with $q_{i}(t)\in {\RR}[t]$ for $i=1,2,3$, and \[n=\mbox{max}\{\deg(q_{1}(t)),\deg(q_{2}(t)),\deg(q_{3}(t))\}.\]Finally, let $f({\bf x})=\bfA {\bf x}+\bfb$, with $\bfA\in {\Bbb R}^3$ orthogonal, $\bfb\in{\Bbb R}^3$, be an involution leaving $S$ invariant, and let $\varphi:\RR^2\to \RR^2$ be the birational transformation making the diagram in Eq. \eqref{eq:fundamentaldiagram} commutative. Then $\varphi(t,s)$ is as in Eq. \eqref{phi-funct-2}, with $\psi(t)$ as in Eq. \eqref{eq:Moebius}, and: 
\begin{itemize}
\item [(I)] $\alpha=-\delta$, $k^2(\gamma \beta+\delta^2)^n=1$, or
\item [(II)] $\varphi(t,s)=(t,-s+c(t))$, with $c(t)$ a rational function. 
\end{itemize}
\end{theorem}

\noindent Observe that in case (II) $f$ fixes each line of the ruling, and acts on these lines as an affine involution.  


\begin{remark} \label{similarities}
Since any similarity can be written as $f({\bf x})=\lambda \bfQ +\bfb$, where $\lambda\neq 0$ is the scaling constant, taking norms in Eq. \eqref{at}, with $k$ a constant, we reach the condition
\begin{equation}\label{eq-compu-2}
\lambda^2\Vert \bfq_1(t)\Vert^2-k^2\cdot (\gamma t+\delta)^{2n_2}\cdot \Vert \bfq_2(\psi(t))\Vert^2=0.
\end{equation} 
The analysis in this case is very similar to that of isometries, although the polynomial system has one more variable, namely $\lambda$. 
\end{remark}

\subsection{Two examples} \label{detailed}

We illustrate the previous ideas in the following examples, one corresponding to the case $r=3$, and the other one to the case $r=2$. 

\vspace{0.3 cm}
\noindent {\bf Example 1.} Let $S_1$ and $S_2$ be the rational ruled surfaces parametrized by $\pmb{x_1}(t,s)=\pmb{p_1}(t)+s\cdot \pmb{q_1}(t)$ and $\pmb{x_2}(t,s)=\pmb{p_2}(t)+s\cdot \pmb{q_2}(t)$ respectively, where
\[
\begin{array}{l}
\bfp_1(t)= (t^4+t^2+t,t^6+t^3,t^5+t^3+t^2+3t),\\
\bfq_1(t)=(t^3+t,t^5,t^4+t^2+3),\\
\bfp_2(t)=(5t^4+5t^2+5t-1,3t^5+3t^3+3t^2+9t+5,-t^6+t^4-t^3+t^2+t),\\
\bfq_2(t)=\left(5t^3+5t,3t^4+3t^2+9,-t^5+t^3+t\right).
\end{array}
\]
In this case, $n=5$. Furthermore, when we write $\bfq_1(t)$ as in Eq. \eqref{qcomponents}, we observe that we fall in the case $r=3$. The surfaces $S_1,S_2$ are shown in Fig. 1.

We consider first the symmetries of $S_1$. If we directly apply Algorithm \ref{equiv-alg}, with $S_1=S_2$, the solubility of the linear system derived from Eq. \eqref{eqmain} in the entries of the matrix $\bfA$ yields, after factoring out $k$, the three following conditions on the coefficients $\alpha, \beta, \gamma, \delta:$ 

\begin{itemize}
	\item [(a)] $2\alpha^3\delta\gamma -3\alpha^2\beta\delta^2+3\alpha^2\beta\gamma^2-6\alpha\beta^2\delta\gamma -4\alpha\delta^3\gamma +4\alpha\delta\gamma^3-\beta^3\gamma^2-6\beta\delta^2\gamma^2+\beta\gamma^4=0.$
	
	\item [(b)] $5\alpha^4\beta -10\alpha^2\beta^3=0.$
	
	\item [(c)] $\alpha^4\delta +4\alpha^3\beta\gamma -6\alpha^2\beta^2\delta -\alpha^2\delta^3+3\alpha^2\delta\gamma^2-4\alpha\beta^3\gamma -6\alpha\beta\delta^2\gamma +2\alpha\beta\gamma^3-3\beta^2\delta\gamma^2-30\delta^3\gamma^2+15\delta\gamma^4=0.$
\end{itemize}

Since $\alpha\delta-\beta\gamma\neq0$, we add the equation $(\alpha\delta-\beta\gamma)u-1=0$. Additionally, from Eq. \eqref{eq-compu} we get the following 11 equations of degree 10 in $\alpha, \beta, \gamma, \delta$ and degree 2 in $k$:

\begin{itemize}
\item $-\beta^{10}k^2-\beta^8\delta^2k^2-3\beta^6\delta^4k^2-9\beta^4\delta^6k^2-7\beta^2\delta^8k^2-9\delta^{10}k^2+9=0.$

\item $-10\alpha\beta^9k^2-8\alpha\beta^7\delta^2k^2-18\alpha\beta^5\delta^4k^2-36\alpha\beta^3\delta^6k^2-14\alpha\beta\delta^8k^2$$-2\beta^8\delta\gamma k^2\\-12\beta^6\delta^3\gamma k^2-54\beta^4\delta^5\gamma k^2-56\beta^2\delta^7\gamma k^2-90\delta^9\gamma k^2=0.$

\item $-45\alpha^2\beta^8k^2-28\alpha^2\beta^6\delta^2k^2-45\alpha^2\beta^4\delta^4k^2-54\alpha^2\beta^2\delta^6k^2-7\alpha^2\delta^8k^2\\-16\alpha\beta^7\delta\gamma k^2-72\alpha\beta^5\delta^3\gamma k^2-216\alpha\beta^3\delta^5\gamma k^2-112\alpha\beta\delta^7\gamma k^2-\beta^8\gamma^2k^2-18\beta^6\delta^2\gamma^2k^2-135\beta^4\delta^4\gamma^2k^2-196\beta^2\delta^6\gamma^2k^2-405\delta^8\gamma^2k^2+7=0.$

\item $-120\alpha^3\beta^7k^2-56\alpha^3\beta^5\delta^2k^2-60\alpha^3\beta^3\delta^4k^2-36\alpha^3\beta\delta^6k^2-56\alpha^2\beta^6\delta\gamma k^2$\\$-180\alpha^2\beta^4\delta^3\gamma k^2-324\alpha^2\beta^2\delta^5\gamma k^2-56\alpha^2\delta^7\gamma k^2-8\alpha\beta^7\gamma^2k^2$\\$-108\alpha\beta^5\delta^2\gamma^2k^2-540\alpha\beta^3\delta^4\gamma^2k^2-392\alpha\beta\delta^6\gamma^2k^2-12\beta^6\delta\gamma^3k^2$\\$-180\beta^4\delta^3\gamma^3k^2-392\beta^2\delta^5\gamma^3k^2-1080\delta^7\gamma^3k^2=0.$

\item $-210\alpha^4\beta^6k^2-70\alpha^4\beta^4\delta^2k^2-45\alpha^4\beta^2\delta^4k^2-9\alpha^4\delta^6k^2-112\alpha^3\beta^5\delta\gamma k^2\\-240\alpha^3\beta^3\delta^3\gamma k^2-216\alpha^3\beta\delta^5\gamma k^2-28\alpha^2\beta^6\gamma^2k^2-270\alpha^2\beta^4\delta^2\gamma^2k^2\\-810\alpha^2\beta^2\delta^4\gamma^2k^2-196\alpha^2\delta^6\gamma^2k^2-72\alpha\beta^5\delta\gamma^3k^2-720\alpha\beta^3\delta^3\gamma^3k^2\\-784\alpha\beta\delta^5\gamma^3k^2-3\beta^6\gamma^4k^2-135\beta^4\delta^2\gamma^4k^2-490\beta^2\delta^4\gamma^4k^2-1890\delta^6\gamma^4k^2+9=0.$

\item $-252\alpha^5\beta^5k^2-56\alpha^5\beta^3\delta^2k^2-18\alpha^5\beta\delta^4k^2-140\alpha^4\beta^4\delta\gamma k^2-180\alpha^4\beta^2\delta^3\gamma k^2-54\alpha^4\delta^5\gamma k^2-56\alpha^3\beta^5\gamma^2k^2-360\alpha^3\beta^3\delta^2\gamma^2k^2-540\alpha^3\beta\delta^4\gamma^2k^2$\\$-180\alpha^2\beta^4\delta\gamma^3k^2-1080\alpha^2\beta^2\delta^3\gamma^3k^2-392\alpha^2\delta^5\gamma^3k^2-18\alpha\beta^5\gamma^4k^2$\\$-540\alpha\beta^3\delta^2\gamma^4k^2-980\alpha\beta\delta^4\gamma^4k^2-54\beta^4\delta\gamma^5k^2-392\beta^2\delta^3\gamma^5k^2$\\$-2268\delta^5\gamma^5k^2=0.$

\item $-210\alpha^6\beta^4k^2-28\alpha^6\beta^2\delta^2k^2-3\alpha^6\delta^4k^2-112\alpha^5\beta^3\delta\gamma k^2-72\alpha^5\beta\delta^3\gamma k^2-70\alpha^4\beta^4\gamma^2k^2-270\alpha^4\beta^2\delta^2\gamma^2k^2-135\alpha^4\delta^4\gamma^2k^2-240\alpha^3\beta^3\delta\gamma^3k^2$\\$-720\alpha^3\beta\delta^3\gamma^3k^2-45\alpha^2\beta^4\gamma^4k^2-810\alpha^2\beta^2\delta^2\gamma^4k^2-490\alpha^2\delta^4\gamma^4k^2\\-216\alpha\beta^3\delta\gamma^5k^2-784\alpha\beta\delta^3\gamma^5k^2-9\beta^4\gamma^6k^2-196\beta^2\delta^2\gamma^6k^2-1890\delta^4\gamma^6k^2+3=0.$

\item $-120\alpha^7\beta^3k^2-8\alpha^7\beta\delta^2k^2-56\alpha^6\beta^2\delta\gamma k^2-12\alpha^6\delta^3\gamma k^2-56\alpha^5\beta^3\gamma^2k^2\\-108\alpha^5\beta\delta^2\gamma^2k^2-180\alpha^4\beta^2\delta\gamma^3k^2-180\alpha^4\delta^3\gamma^3k^2-60\alpha^3\beta^3\gamma^4k^2$\\$-540\alpha^3\beta\delta^2\gamma^4k^2-324\alpha^2\beta^2\delta\gamma^5k^2-392\alpha^2\delta^3\gamma^5k^2-36\alpha\beta^3\gamma^6k^2$\\$-392\alpha\beta\delta^2\gamma^6k^2-56\beta^2\delta\gamma^7k^2-1080\delta^3\gamma^7k^2=0.$

\item $-45\alpha^8\beta^2k^2-\alpha^8\delta^2k^2-16\alpha^7\beta\delta\gamma k^2-28\alpha^6\beta^2\gamma^2k^2-18\alpha^6\delta^2\gamma^2k^2\\-72\alpha^5\beta\delta\gamma^3k^2-45\alpha^4\beta^2\gamma^4k^2-135\alpha^4\delta^2\gamma^4k^2-216\alpha^3\beta\delta\gamma^5k^2-54\alpha^2\beta^2\gamma^6k^2-196\alpha^2\delta^2\gamma^6k^2-112\alpha\beta\delta\gamma^7k^2-7\beta^2\gamma^8k^2-405\delta^2\gamma^8k^2+1=0.$

\item $-10\alpha^9\beta k^2-2\alpha^8\delta\gamma k^2-8\alpha^7\beta\gamma^2k^2-12\alpha^6\delta\gamma^3k^2-18\alpha^5\beta\gamma^4k^2-54\alpha^4\delta\gamma^5k^2-36\alpha^3\beta\gamma^6k^2-56\alpha^2\delta\gamma^7k^2-14\alpha\beta\gamma^8k^2-90\delta\gamma^9k^2=0.$

\item $-\alpha^{10}k^2-\alpha^8\gamma^2k^2-3\alpha^6\gamma^4k^2-9\alpha^4\gamma^6k^2-7\alpha^2\gamma^8k^2-9\gamma^{10}k^2+1=0.$
\end{itemize}

The union of the above 11 conditions plus the three conditions (a), (b), (c), with the additional condition $(\alpha\delta-\beta\gamma)u-1=0$, provides four tentative solutions for $\alpha,\beta,\gamma,\delta,k$, corresponding to

\begin{equation}\label{onlysol}
\{\alpha = \pm 1, \beta = 0,\gamma=0, \delta=1, k = \pm 1\}.
\end{equation}

Then we compute (numeric) values for $\bfA,\bfb$ and explicit forms for $c(t)$. Finally, using Eq. \eqref{fundam-eq} as a test, we get one nontrivial symmetry for $S_1$, corresponding to $\{\alpha=- 1, \beta=0, \delta=1, \gamma=0,k=1\}$, with $\varphi(t,s)=(-t,s+2t)$. The symmetry is defined by $f_0({\bf x})=\bfA_0 {\bf x}+\bfb_0$, where
	\begin{equation}\label{Q-example-1}
	\begin{array}{cc}\bfA_0=\left( \begin{array}{ccr}
		-1 & 0 & 0\\		
		0 & -1 & 0 \\
		0 & 0 & 1\\
		\end{array}\right), &
		\bfb_0=\left( \begin{array}{ccc}
		0 &		
	    0 &
		0
		\end{array}\right)^T,
		\end{array}
		\end{equation}
so $S_1$ is symmetric with respect to the $z$-axis. The total time of computation here is $3.744$ seconds. A similar timing, namely $4.602$ seconds, is obtained if, instead of using the above 11 conditions, one considers the polynomial system consisting of (a), (b), (c) and the six conditions (of degree 10) derived from the fact that $\bfA$ is orthogonal, after expressing the entries of $\bfA$ in terms of $\alpha,\beta,\gamma,\delta$. In that case, one obtains tentative possibilities depending on $k$, namely $\{\alpha=\pm 1, \beta=0, \gamma=0, \delta=1, k=k\}$, so the corresponding expressions for $\bfb,c(t)$ depend on $k$ as well. Finally, using Eq. \eqref{fundam-eq}, the symmetry is computed.  

\vspace{0.3 cm}
In order to check whether or not $S_1,S_2$ are affinely equivalent, we apply {\tt Algorithm 1}. The solubility conditions of the linear systems derived from Eq. \eqref{eq} in the entries of the matrix $\bfA$ are:

 \begin{itemize}
 	\item $10\alpha^3\delta\gamma-15\alpha^2\beta\delta^2+15\alpha^2\beta\gamma^2-30\alpha\beta^2\delta\gamma-20\alpha\delta^3\gamma+20\alpha\delta\gamma^3-5\beta^3\gamma^2-30\beta\delta^2\gamma^2+5\beta\gamma^4=0.$
 	
 	\item $3\alpha^4\delta+12\alpha^3\beta\gamma-18\alpha^2\beta^2\delta-3\alpha^2\delta^3+9\alpha^2\delta\gamma^2-12\alpha\beta^3\gamma-18\alpha\beta\delta^2\gamma+6\alpha\beta\gamma^3-9\beta^2\delta\gamma^2-90\delta^3\gamma^2+45\delta\gamma^4=0.$
 	
 	\item $-5\alpha^4\beta+2\alpha^3\delta\gamma+10\alpha^2\beta^3-3\alpha^2\beta\delta^2+3\alpha^2\beta\gamma^2-6\alpha\beta^2\delta\gamma-4\alpha\delta^3\gamma+4\alpha\delta\gamma^3-\beta^3\gamma^2-6\beta\delta^2\gamma^2+\beta\gamma^4=0.$
 \end{itemize}
  
Since $\alpha\delta-\beta\gamma\neq0$, we add the equation $(\alpha\delta-\beta\gamma)u-1=0$, obtaining expressions for $\alpha,\beta,\gamma,\delta$ only depending on $k$; the same thing happens with $\bfb$ and $c(t)$. Finally, we get two $\varphi$s corresponding to affine equivalences, namely
\[
\varphi_1(t,s)=(t,s),\mbox{ }\varphi_2(t,s)=(-t,s+2t).
\]
The mapping $\varphi_1(t,s)$ corresponds to the affine mapping $f_1({\bf x})=\bfA_1 {\bf x}+\bfb_1$, where 
\begin{equation}\label{Q1-example}
	\begin{array}{cc}\bfA_1=\left( \begin{array}{ccr}
		5 & 0 & 0\\		
		0 & 0 & 3 \\
		1 & -1 & 0\\
		\end{array}\right), &
		\bfb_1=\left( \begin{array}{ccc}
		-1 &		
	    5 &
		0
		\end{array}\right)^T.
		\end{array}
		\end{equation}

The mapping $\varphi_2(t,s)$ corresponds to the affine mapping $f_2({\bf x})=\bfA_2 {\bf x}+\bfb_2$, where
\begin{equation}\label{Q-example}
	\begin{array}{cc}\bfA_2=\left( \begin{array}{ccr}
		-5 & 0 & 0\\		
		0 & 0 & 3 \\
		-1 & 1 & 0\\
		\end{array}\right), &
		\bfb_2=\left( \begin{array}{ccc}
		-1 &		
	    5 &
		0
		\end{array}\right)^T.
		\end{array}
		\end{equation}

Therefore, we conclude that $S_1,S_2$ are related by two affine mappings $f_1,f_2$. Notice that this result is coherent with the fact that $S_1$ has a non-trivial symmetry; in fact, one can check that $f_2=f_1\circ f_0$. The computation time was 5.179 seconds. 

\begin{figure}
\begin{center}
$$\begin{array}{cc}
\vspace{-1.5cm}\includegraphics[scale=0.25]{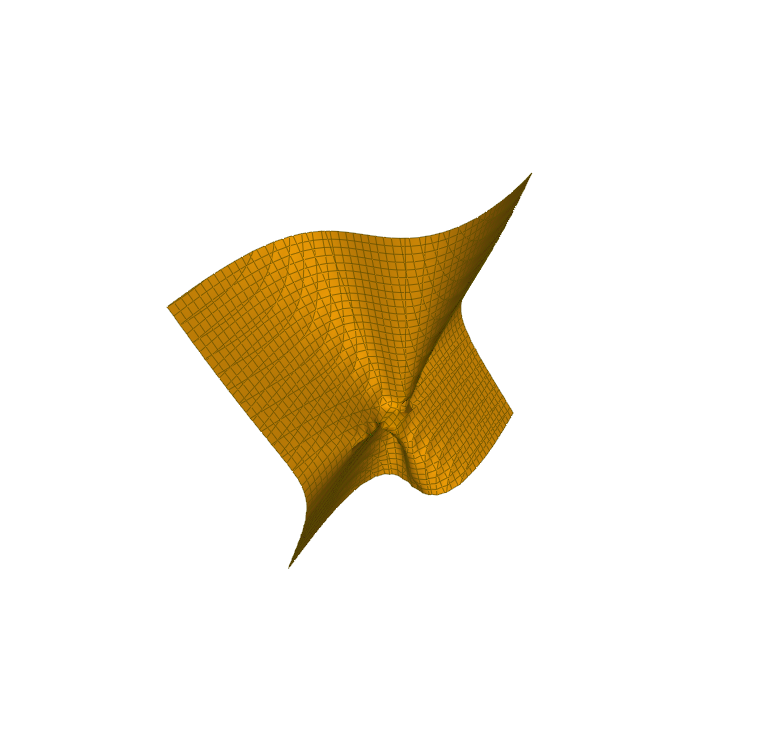} & \hspace{-1.8cm}\includegraphics[scale=0.25]{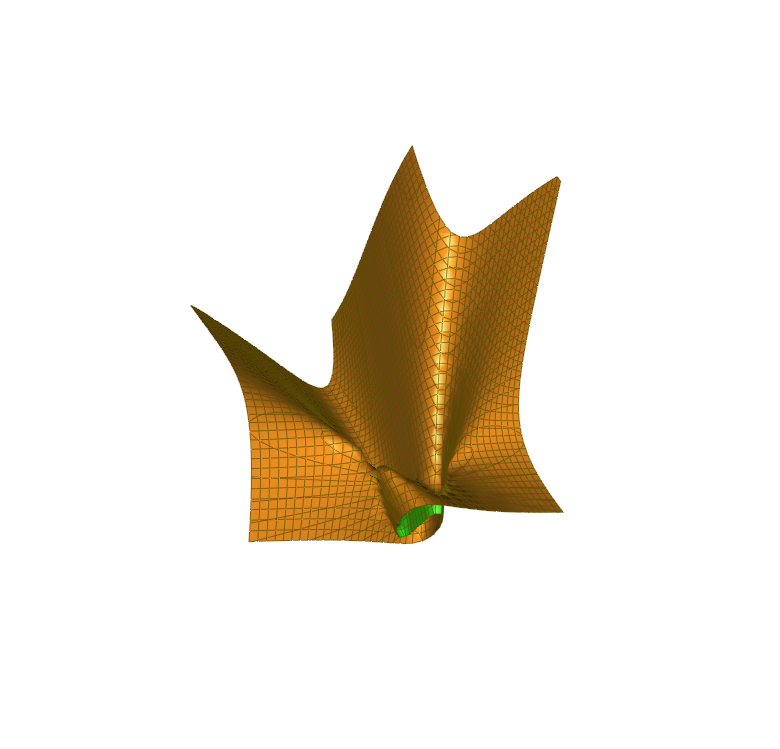}
\end{array}$$
\end{center}
\caption{$S_1$ (left) and $S_2$ (right).}\label{ej1}
\end{figure}
	
\vspace{0.4 cm}

\noindent{\bf Example 2.} Let $S_1$ and $S_2$ be the rational ruled surfaces parametrized by $\pmb{x_1}(t,s)=\pmb{p_1}(t)+s\cdot \pmb{q_1}(t)$ and $\pmb{x_2}(t,s)=\pmb{p_2}(t)+s\cdot \pmb{q_2}(t)$ respectively, where
\[
\begin{array}{l}
\bfp_1(t)= \left(t+\dfrac{3}{4},4t^2+3,t\right),\\ [0.4cm]
\bfq_1(t)=(t^3+2t^2+1,-t^3+t^2+t,-t^3+t^2+t),\\[0.2cm]
\bfp_2(t)=\left(\dfrac{\sqrt{3}+1}{2}t+\dfrac{3\sqrt{3}}{8}-\dfrac{1}{2},4t^2+5,\dfrac{\sqrt{3}-1}{2}t-\dfrac{\sqrt{3}}{2}-\dfrac{3}{8}\right),
\end{array}
\]
and
{\begin{multline*}	
\label{q2-example2}
				\bfq_2(t)=\left(\dfrac{\sqrt{3}-1}{2}t^3+\left(\sqrt{3}+\dfrac{1}{2}\right)t^2+\dfrac{t}{2}+\dfrac{\sqrt{3}}{2},-t^3+t^2+t,\right.\\\left.-\left(\dfrac{\sqrt{3}+1}{2}\right)t^3+\left(\dfrac{\sqrt{3}}{2}-1\right)t^2+\dfrac{\sqrt{3}}{2}t-\dfrac{1}{2}\right),
	\end{multline*}}

\noindent Here, $n=3$. Furthermore, when we write $\bfq_1(t)$ as in Eq. \eqref{qcomponents}, we observe that we are in the case $r=2$.

In this case, we analyze the isometries mapping $S_1$ onto $S_2$. There is only one isometry, associated with $\varphi(t,s)=(t,s)$, defined by $f({\bf x})=\bfA {\bf x}+\bfb$, where
	\begin{equation}\label{Q2-example-1}
	\begin{array}{cc}\bfA=\left( \begin{array}{ccr}
	\dfrac{\sqrt{3}}{2} & 0 & \dfrac{1}{2}\\		
	0 & 1 & 0 \\
	-\dfrac{1}{2} & 0 & \dfrac{\sqrt{3}}{2}\\
	\end{array}\right), &
	\bfb=\left( \begin{array}{ccc}
	-\dfrac{1}{2} &		
	2 &
	-\dfrac{\sqrt{3}}{2}
	\end{array}\right)^T,
	\end{array}
	\end{equation}
corresponding to a rotation of $\dfrac{\pi}{6}$ around the $y$-axis. In order to compute this isometry, applying Algorithm 1 with the additional equations corresponding to Eq. \eqref{eq-compu} we only need to test two tentative solutions, and the computation time is $3.588$ seconds. If, instead of Eq. \eqref{eq-compu}, we use the orthogonality conditions on the columns of the matrix $\bfA$, we need to test four tentative solutions, and the computation time is $4.696$ seconds.


\section{Experimentation and performance of the method.}\label{sec-exp}

We have implemented the method described in Section \ref{sec-symmetries} in the computer algebra system Maple 18, and we have tried several examples in an Intel(R) Core(TM) 2, Quad CPU Q6600, with 2.40 GHz and 4 Gb RAM; this is also the machine used in the examples of Subsection \ref{detailed}. We have analyzed both affine equivalences, and isometries. In the case of isometries, in the computation we include the conditions derived from Eq. \eqref{eq-compu}, since we observe that they highly speed up the computation. 

The results for affine equivalences of some representative examples are summarized in Tables 1 and 2. The surface $S_1$ is given by the parametrizacion $\tilde{\bfx}_i$ and $S_2$ is given by the parametrizacion $\tilde{\bfy}_i$, both shown in the table. When the surfaces are affinely equivalent, $\tilde{\bfy}_i$ is the result of applying on $\tilde{\bfx}_i$ an affine equivalence with matrix
$$\left( \begin{array}{ccr}
-1/2 & -1 & 0\\		
0 & 1 & 1 \\
0 & 2 & 3\\
\end{array}\right).$$
For each example, we have included: (1) a picture of the surface defined by $\tilde{\bfx}_i$ (or the surfaces defined by $\tilde{\bfx}_i,\tilde{\bfy}_i$, when they are not affinely equivalent); (2) the degree $N$ (``deg") of the parametrizations, i.e. the maximum power of $t$ appearing in the numerators and denominators of $\bfp_i(t),\bfq_i(t)$; (3) the computation time (in seconds) of the method for all the affine mappings, and the computation time using the implicit equation of the surface. 

The examples of Table 1 and Table 2 with more than one affine equivalence correspond to surfaces with symmetries. Furthermore, in some cases we identify infinitely many equivalences, implying that the surfaces are invariant under infinitely many affine mappings. In the column of timings, we highlight in red the worst time between our method, and the naive method mentioned in the Introduction using the implicit equation. This last timing does not include the time of computing the implicit equation, i.e. we assume that the implicit equation is already known. Only in one of the examples shown, where the implicit equation is very simple ($F(x,y,z)=x^3-27yz^2$) the method using the implicit equation is faster. 

The results for symmetries and isometries of several representative examples are summarized in Tables 3 and 4: for each example, we include the same data as in the affine equivalences table, plus the computation time (in seconds) of our method for computing all the symmetries of the surface given by $\bfx_i$ (``all sym."), for computing only the involutions of the surface (``involutions"), and for computing the isometries (``isometries") between each surface and its image under an orthogonal transformation with associated matrix 

\begin{equation}\label{MatrixIso}
\left( \begin{array}{ccr}
0 & 1 & 0\\		
4/5 & 0 & -3/5 \\
3/5 & 0 & 4/5\\
\end{array}\right).
\end{equation}

We also include the computation time using the implicit equation of the surface (``implicit"), assuming that this equation is available. Maple was able to provide an answer with this last method in less than 90 seconds in only two of the examples. In one of them, as we observed before, the implicit equation turned out to be very simple, which explains why the method using the implicit equation is faster. Finally, we include the type of symmetries found, too. In some cases, the symmetries detected are composites of rotations and reflections, denoted as ``rotation+reflection".

\newgeometry{left=3cm, right=3cm}
\begin{table}[h!]
	
	\begin{tabular*}{\columnwidth}{lccr}
		\toprule
		 Picture of $S_1$ & deg. & \hspace{1cm} computation time (secs.)/imp. & Affine equivalences \\
		\midrule
\raisebox{-3em}{\includegraphics[scale=0.21]{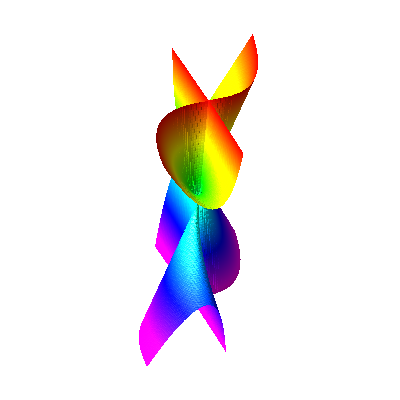}}	& 5 &\hspace{1cm} 2.309\,/\,\textcolor{red}{$> 90$} & 2 \\[0.3cm] 
		\multicolumn{4}{c}{\footnotesize$\tilde{\bfx}_1(t,s)=\left(\dfrac{t^2}{t^2+1},\dfrac{t^4}{t^2+1},\dfrac{t^5}{t^2+1}\right)+s\cdot(t,t^3,1)$}\\[0.3cm]
		\multicolumn{4}{c}{\footnotesize$\tilde{\bfy}_1(t,s)=\left(-\dfrac{1}{2}\dfrac{t^2(2t^2+1)}{t^2+1},\dfrac{t^4(t+1)}{t^2+1},\dfrac{t^4(3t+2)}{t^2+1}\right)+s\left(-\dfrac{1}{2}(2t^2+1)t,(t+1)(t^2-t+1),2t^3+3\right)$}\\
			\midrule
		\raisebox{-3em}{\includegraphics[scale=0.17]{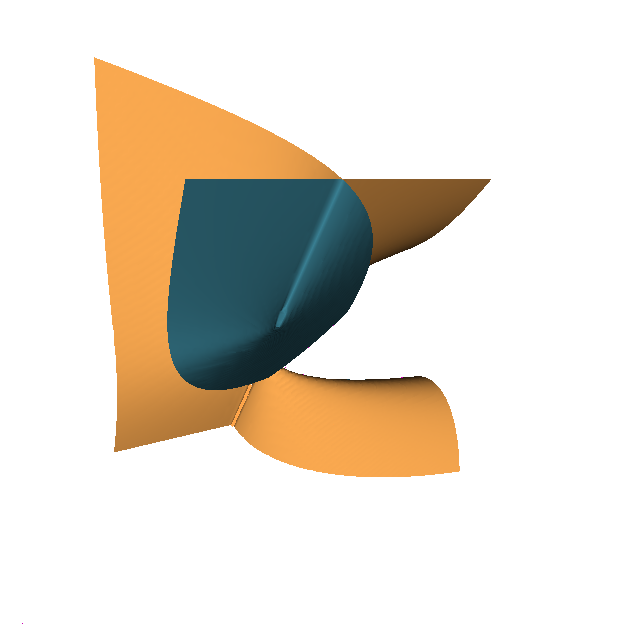}}	& 2 & \hspace{1cm} 16.676\,/\,\textcolor{red}{$> 90$} & $\infty$ \\[0.3cm] 
		\multicolumn{4}{c}{\footnotesize$\tilde{\bfx}_2(t,s)=(4,1,t)+s\cdot((t+1)^2,t+1,1)$}\\[0.2cm]
		\multicolumn{4}{c}{\footnotesize$\tilde{\bfy}_2(t,s)=\left(-3,t+1,3t+2\right)+s\left(-\dfrac{1}{2}(t+3)(t+1),t+2,2t+5\right)$}\\
		\midrule
	\raisebox{-3em}{\includegraphics[scale=0.18]{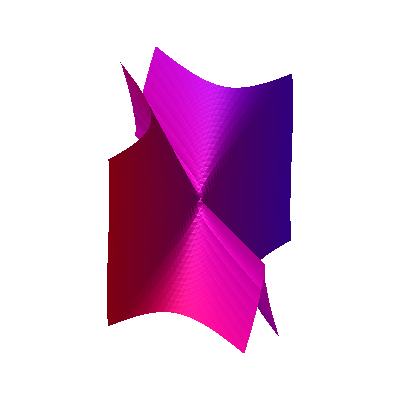}}	& 3 &\hspace{1cm} \textcolor{red}{2.870}\,/\,$0.687$ & $\infty$ \\[0.3cm] 
		\multicolumn{4}{c}{\footnotesize$\tilde{\bfx}_3(t,s)=s\cdot(3(t+1)^2(t-1),(t-1)^3,(t+1)^3)$}\\[0.2cm]
		\multicolumn{4}{c}{\footnotesize$\tilde{\bfy}_3(t,s)=s\left(-\frac{1}{2}(t-1)(5t^2+2t+5),2t(t^2+3),5t^3+3t^2+15t+1\right)$}\\
		\midrule
	\hspace{-0.5cm}\raisebox{-5em}{\includegraphics[scale=0.18]{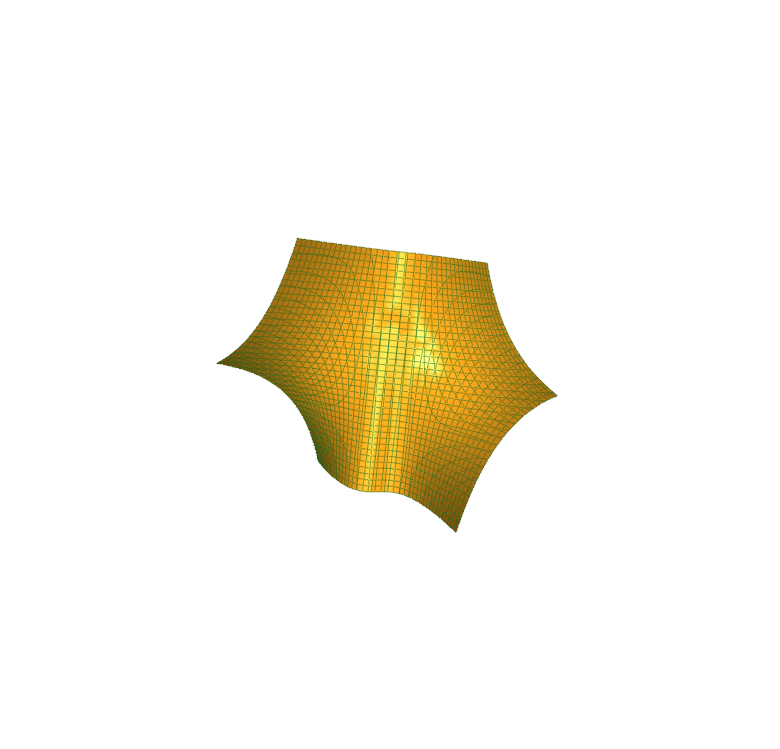}}	& 3 & \hspace{1cm} 3.167\,/\,\textcolor{red}{$> 90$} & 2 \\[-0.5cm] 
		\multicolumn{4}{c}{\footnotesize$\tilde{\bfx}_4(t,s)=(t^3+t,t^2-3,t^3+t)+s\cdot(-t^3,t^2-8,2t^3-3)$}\\[0.1cm]
		\multicolumn{4}{c}{\footnotesize$\tilde{\bfy}_4(t,s)=\left(\frac{t^3}{2}-\frac{t}{2}-t^2+3,t^3+t^2+t-3,3t^3+2t^2+3t-6\right)+s\left(\frac{t^3}{2}-t^2+8,2t^3+t^2-t-8,6t^3+2t^2-3t-16\right)$}\\
	
		\bottomrule
		
	\end{tabular*}
	\caption{}
\end{table}
\restoregeometry

\newgeometry{left=3cm, right=3cm}
\begin{table}[h!]
	
	\begin{tabular*}{\columnwidth}{lccc}
		\toprule
		Picture of $S_1$ & deg. &\hspace{-1cm} Computation time (secs.)/imp. & \hspace{-1cm} Affine equivalences \\
		\midrule
		
		\raisebox{-3em}{\includegraphics[scale=0.18]{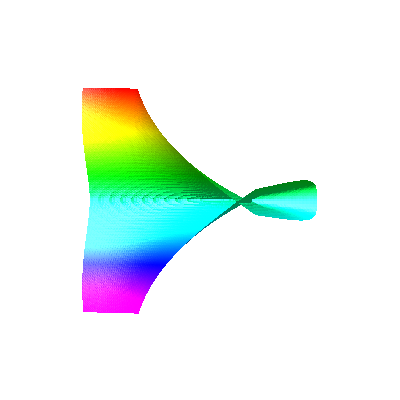}}	& 4 &  0.452\,/\,\textcolor{red}{$> 90$} &  2 \\[-0.5cm] 
		& \multicolumn{3}{l}{\footnotesize$\tilde{\bfx}_5(t,s)=(t^4+t^2,t^2-3,t^3+t)+s\cdot(-t^4+2,t^2-8,2t)$}\\
		[0.1cm]
		\multicolumn{4}{c}{\footnotesize$\tilde{\bfy}_5(t,s)=\left(-\frac{t^4}{2}-\frac{3t^2}{2}+3,t^3+t^2+t-3,3t^3+2t^2+3t-6\right)+s\left(\frac{t^4}{2}-t^2+7,t^2+2t-8,2t^2+6t-16\right)$}\\
		\midrule
		
	\hspace{-0.7cm}\raisebox{-8em}{\includegraphics[scale=0.18]{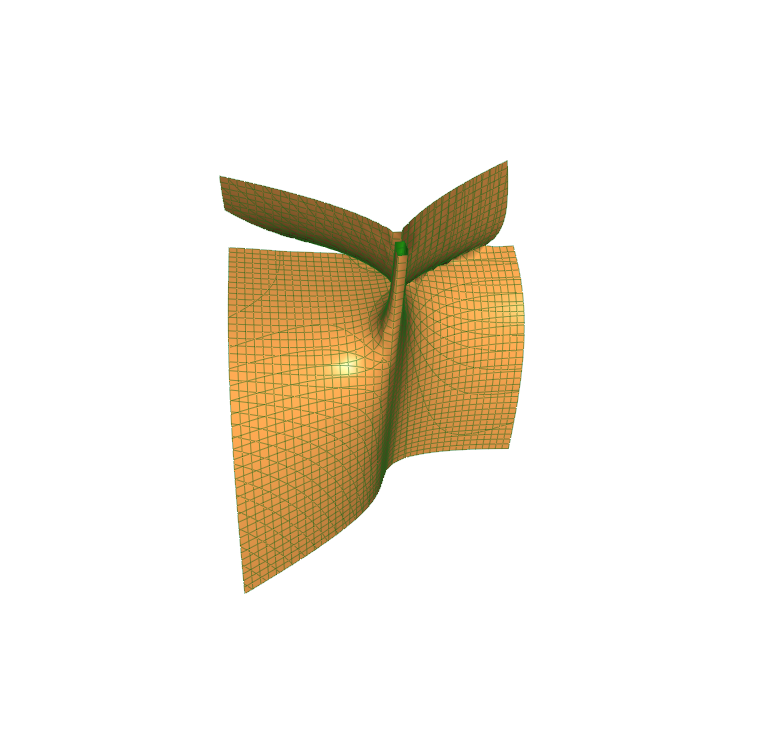}}	& 3 &  6.583\,/\,\textcolor{red}{$> 90$} & 1 \\[-2cm] 
		& \multicolumn{2}{c}{\footnotesize$\tilde{\bfx}_6(t,s)=\left(\frac{t^3+t}{t^2+1},-3,t+t^2\right)+s\cdot(-t^3,t^2-8,t^2-t)$} & \\[0.3cm]
		\multicolumn{4}{c}{\footnotesize$\tilde{\bfy}_6(t,s)=\left(3-\frac{t}{2},t^2+t-3,3t^2+3t-6\right)+s\cdot\left(-t^2+8+\frac{t^3}{3},2t^2-t-8,5t^2-3t-16\right)$}\\
		\midrule
		
		\vspace{0.4cm}\hspace{0.1cm}\raisebox{-5em}{\includegraphics[scale=0.17]{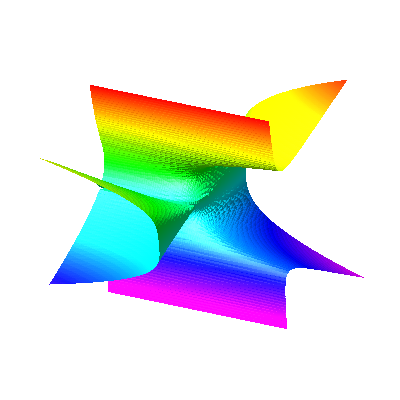}}	& 4 &  0.608\,/\,\textcolor{red}{$> 90$} & 2 \\[-1.3cm] 
		& \multicolumn{2}{c}{\footnotesize$\tilde{\bfx}_7(t,s)=\left(\frac{t^4+t^2}{t^2+3},\frac{t^2-3}{t^2+3},\frac{t^3+t}{t^2+3}\right)+s\cdot(-t^4+2,t^2-8,2t)$} & \\[0.3cm]
		\multicolumn{4}{c}{\footnotesize$\tilde{\bfy}_7(t,s)=\left(-\frac{1}{2}\cdot\frac{t^2(t^2+5)}{t^2+3},\frac{t^3+t^2+t-3}{t^2+3},\frac{3t^3+6t^2+3t+6}{t^2+3}\right)+s\cdot\left(\frac{1}{2}t^4-t^2+7,t^2+2t-8,2t^2+6t-16\right)$}\\
		
		\midrule
		 
		 \vspace{0.6cm}\hspace{-0.5cm}\raisebox{-5em}{\includegraphics[scale=0.18]{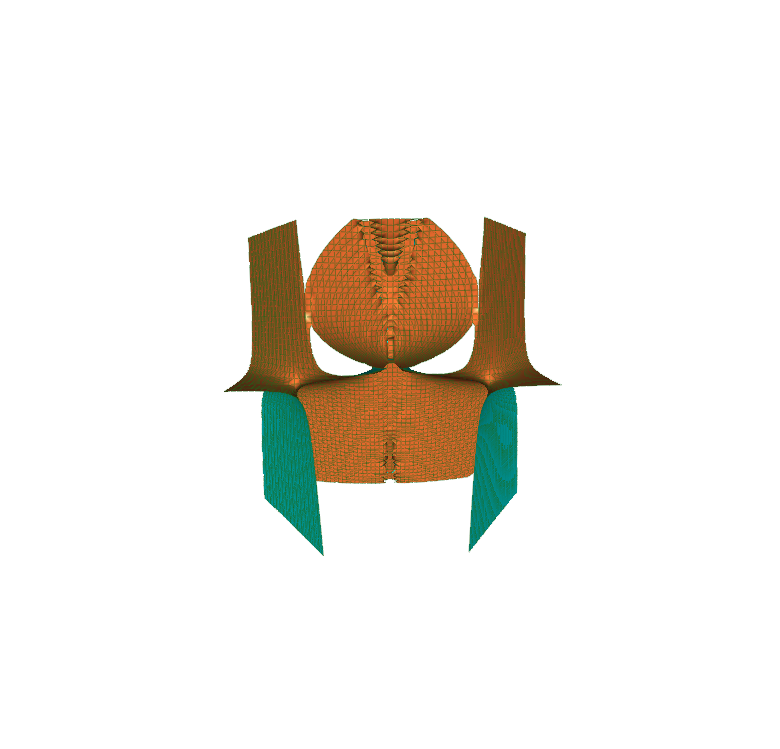}}	& 7 &  11.563 \,/\,\textcolor{red}{$> 90$} & 2 \\[-1.3cm] 
		 & \multicolumn{2}{c}{\footnotesize$\tilde{\bfx}_8(t,s)=\left(t^6-6t^4+t^2+2t,-t^7+6t^5-t^3+t^2+t,t^3+t\right)$} & \\
		 & \multicolumn{2}{c}{\footnotesize$+s\cdot\left(t^5-6t^3+t,-t^6+6t^4-t^2+1,t^2+1\right)$} & \\[0.3cm]
		 \multicolumn{4}{c}{\footnotesize$\tilde{\bfy}_8(t,s)=\left(t^3-\frac{3}{2}t^2-2t-\frac{1}{2}t^6+3t^4+t^7-6t^5,-t^7+6t^5+t^2+2t,-2t^7+12t^5+t^3+2t^2+5t\right)$} \\
		 \multicolumn{4}{c}{\footnotesize$+s\cdot\left(3t^3-\frac{1}{2}t+t^2-1+t^6-6t^4-\frac{1}{2}t^5,-t^6+6t^4+2,-2t^6+12t^4+t^2+5\right)$} \\
		 \midrule
		 
		\multirow{3}{*}{\includegraphics[scale=0.19]{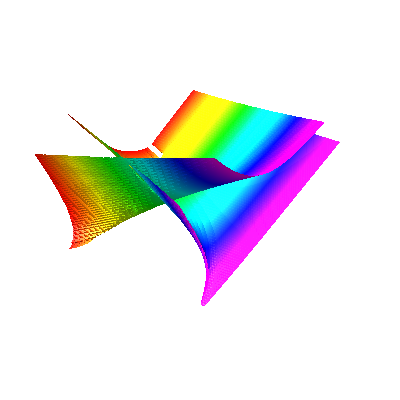} \raisebox{0.5em}{\hspace{-0.5cm}\includegraphics[scale=0.16]{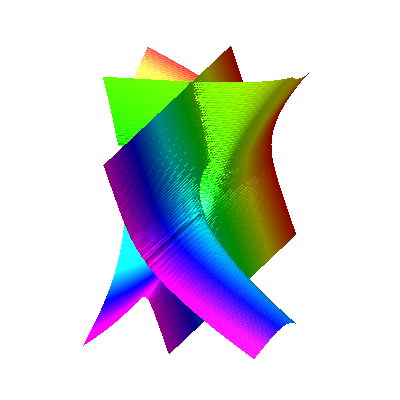}}} & & & \\ & & & \\
		& 4 & 0.469\,/\,\textcolor{red}{$7.800$} & 0 \\[1.5cm] 
		\hspace{0.9cm} $S_1$ \hspace{1.9cm} $S_2$ & \multicolumn{3}{l}{\footnotesize$\tilde{\bfx_9}(t,s)=\left(\frac{1}{t^4},t,1\right)+s\cdot(1,t,t^3)$}\\[0cm]
		\multicolumn{4}{c}{\hspace{0.2cm}\footnotesize$\tilde{\bfy_9}(t,s)=\left(-\frac{2}{t^4},4,3t\right)+s\left(\dfrac{1}{2},t^3+t,3\right)$}\\
		
		\bottomrule
		
	\end{tabular*}
	\caption{}
\end{table}
\restoregeometry

\newgeometry{left=3cm, right=3cm}

\begin{table}
	
	\begin{tabular*}{\columnwidth}{l@{\extracolsep{\stretch{1}}}*{5}{c}}
		\toprule
		parametrization & picture & deg. & \hspace{0.5cm} computation time (secs.) & symmetries \\
	    &  &  & \hspace{0.65cm}all sym.\,/\,involutions\,/ &and isometries \\
	    
	    &  &  & \hspace{0.65cm}isometries\,/ implicit & \\
		\midrule
		&\multirow{3}{*}{\includegraphics[scale=0.09]{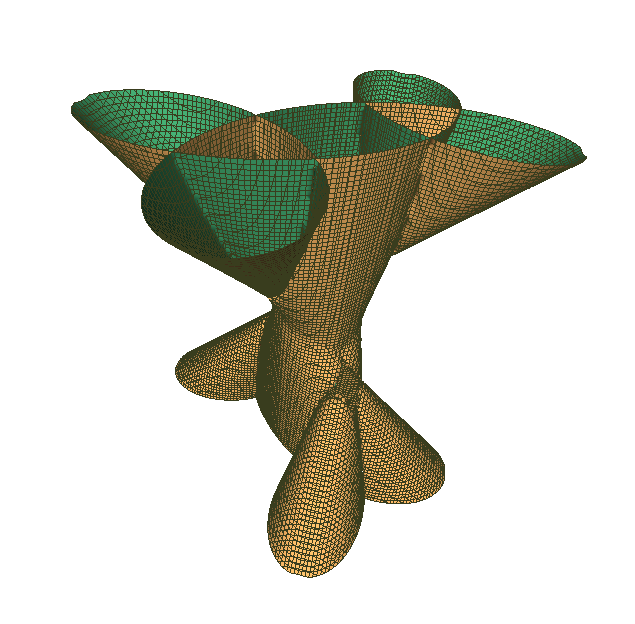}} & & &\\
		& & & & 3 axial \\
		$\bfx_1(t,s)$ & & 9 &\hspace{0.5cm} 9.640\,/\,7.410\,/ &  2 mirror\\ & & & \hspace{0.5cm} 9.267\,/\textcolor{red}{$>90$} & 2 rotational + reflect. \\& & & & 8 isometries \\[0.2cm]
		\multicolumn{5}{c}{\footnotesize{$\bfx_1(t,s)=\left( \frac{2t^8-10t^6-10t^4+5t^2+1}{t^2+1},-\frac{t^9-6t^7+6t^3+t^2-3t+1}{t^2+1},t^7+3t^5+3t^3+t+5\right)$}}\\
		\multicolumn{5}{c}{\footnotesize{$+s\cdot(2t(t^4-6t^2+1),-t^6+7t^4-7t^2+1,(t^2+1)^3)$}}\\
		\midrule
		$\bfx_2(t,s)$
		& \raisebox{-3em}{\includegraphics[scale=0.18]{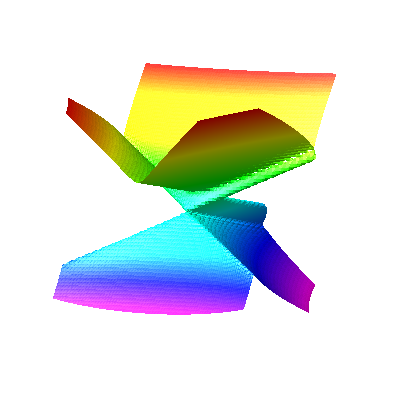}}
		& 7 & \hspace{0.5cm}\vspace{-0.9cm}1.981\,/\,1.812\,/ & 1 reflect.\\
		& & &\vspace{0.5cm}\hspace{0.5cm }1.996\,/\textcolor{red}{$>90$} & 2 isometries \\
		\multicolumn{5}{c}{\small{$\bfx_2(t,s)=\left( \dfrac{t^7+7t^5+3t^3-t^2-3t+1}{t^2+1},\dfrac{2t(4t^5+4t^3+1)}{t^2+1},t(t^2+1)^2\right)+s\cdot(-t^4-6t^2+3,8t^3,(t^2+1)^2)$}}\\
		\midrule
		$\bfx_3(t,s)$
		& \raisebox{-3em}{\includegraphics[scale=0.17]{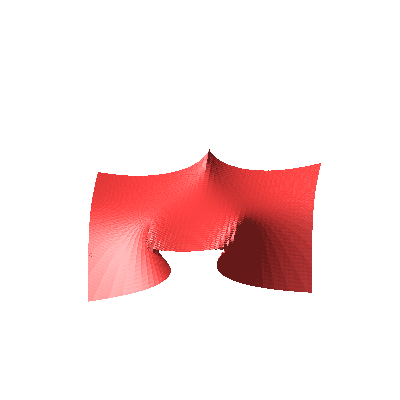}}
		& 7 & \hspace{0.7cm}\vspace{-0.9cm}1.888\,/\,1.778\,/ & 1 reflect. \\
		& & &\vspace{0.5cm} \hspace{0.5cm} 2.184 \,/\textcolor{red}{$>90$}& 2 isometries \\
		\multicolumn{5}{c}{\small{$\bfx_3(t,s)=(t^6-6t^4+t^2+2t,-t^7+6t^5-t^3+t^2+t,t^3+t)+s\cdot(t^5-6t^3+t,-t^6+6t^4-t^2+1,t^2+1)$}}\\
		\midrule
		$\bfx_4(t,s)$
		& \raisebox{-3em}{\includegraphics[scale=0.21]{reg8a}}			& 5 & \hspace{0.7cm}\vspace{-0.9cm}1.684\,/\,1.607\,/ & 1 reflect. \\
		& & &\vspace{0.5cm} \hspace{0.5cm} 2.043\,/\textcolor{red}{$>90$} & 2 isometries \\
		\multicolumn{5}{c}{$\bfx_4(t,s)=\left( \dfrac{t^2}{t^2+1},\dfrac{t^4}{t^2+1},\dfrac{t^5}{t^2+1}\right)+s\cdot(t,t^3,1)$}\\
		\midrule
		&\multirow{6}{*}{\includegraphics[scale=0.18]{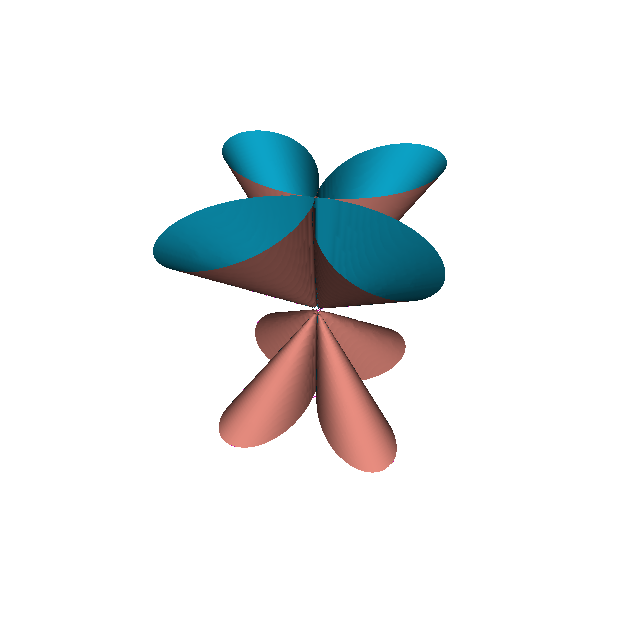}} & & & 5 reflect. \\
		& & & & 5 axial sym. \\ & & &\hspace{0.5cm} 4.587\,/\,3.291\,/ & 1 central \\
		$\bfx_5(t,s)$ & & 6 &\hspace{0.5cm} 10.280\,/\textcolor{red}{$>90$} & 2 rotational sym.  \\
		& & &  & 2 rotational +  reflect. \\  & & & & 16 isometries \\[0.2cm]
		\multicolumn{5}{c}{$\bfx_5(t,s)=s\cdot( 2t(t^4-6t^2+1),(-t^2+1)(t^4-6t^2+1),(t^2+1)^3)$}\\
		\bottomrule
		
	\end{tabular*}
	\caption{}
	\end{table}

\begin{table}
	
	\begin{tabular*}{\columnwidth}{cc@{\extracolsep{\stretch{1}}}cccc}
		\toprule
		parametrization & picture & deg. &\hspace{0.3cm} computation time (secs.) & symmetries\\
		&  &  & \hspace{0.5cm}all sym.\,/\,involutions & and isometries\\
		 &  &  & \hspace{0.65cm}isometries\,/implicit & \\
		\midrule
		$\bfx_6(t,s)$
		& \raisebox{-3em}{\includegraphics[scale=0.17]{reg16b}}			& 2 &\hspace{0.3cm} \vspace{-0.9cm}3.448\,/\,0.390\,/ & 1 axial sym. \\
		& & &\vspace{0.5cm}\hspace{0.3cm} 3.978\,/\textcolor{red}{$63.648$}& 2 isometries\\
		\multicolumn{5}{c}{$\bfx_6(t,s)=(4,1,t)+s\cdot((t+1)^2,t+1,1)$}\\
		\midrule
		& \multirow{6}{*}{\includegraphics[scale=0.18]{cubic}} & & &\\  & & & & central\\
		$\bfx_7(t,s)$ &	& 3 &\hspace{0.3cm} \textcolor{red}{1.451}\,/\,1.185\,/ &  1 reflection \\ & & &\hspace{0.2cm} 2.901\,/0.296& 1 axial sym. \\  & & & & 4 isometries\\ & & & &\\
		\multicolumn{5}{c}{$\bfx_7(t,s)=s\cdot(3(t+1)^2(t-1),(t-1)^3,(t+1)^3)$}\\
		\midrule
		$\bfx_8(t,s)$
		& \raisebox{-3em}{\includegraphics[scale=0.19]{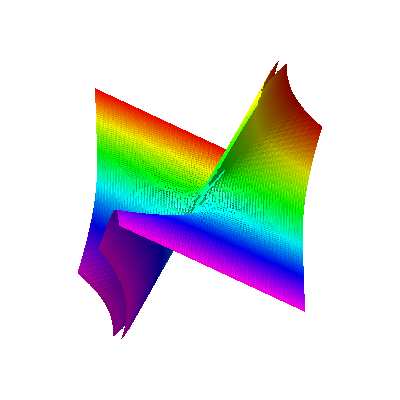}}
		& 7 &\hspace{0.3cm}\vspace{-0.9cm}\hspace{0.3cm} 1.935\,/\,1.809\,/ & central \\
		& & &\vspace{0.5cm}\hspace{0.5cm} 2.372\,/\textcolor{red}{$>90$}& 2 isometries \\
		\multicolumn{5}{c}{$\bfx_8(t,s)=\left( \dfrac{t^3}{t^2+1},\dfrac{t^5}{t^2+1},\dfrac{t^7}{t^2+1}\right)+s\cdot(-t^5+t,3t^7,-2t^3)$}\\	
		\midrule
		$\bfx_9(t,s)$
		& \raisebox{-3em}{\includegraphics[scale=0.18]{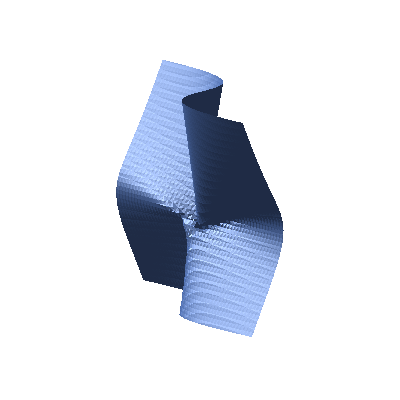}}
		& 6 &\hspace{0.3cm}\vspace{-0.9cm}\hspace{0.4cm} 1.716\,/\,1.653\,/ & 1 axial \\
		& & &\vspace{0.5cm}\hspace{0.7cm}2.200\,/\textcolor{red}{$>90$}& 2 isometries \\
		\multicolumn{5}{c}{$\bfx_9(t,s)=(t^4+t^2+t,t^6+t^3,t^5+t^3+t^2+3t)+s\cdot(t^3+t,t^5,t^4+t^2+3)$}\\
		\midrule
		&\multirow{3}{*}{\includegraphics[scale=0.15]{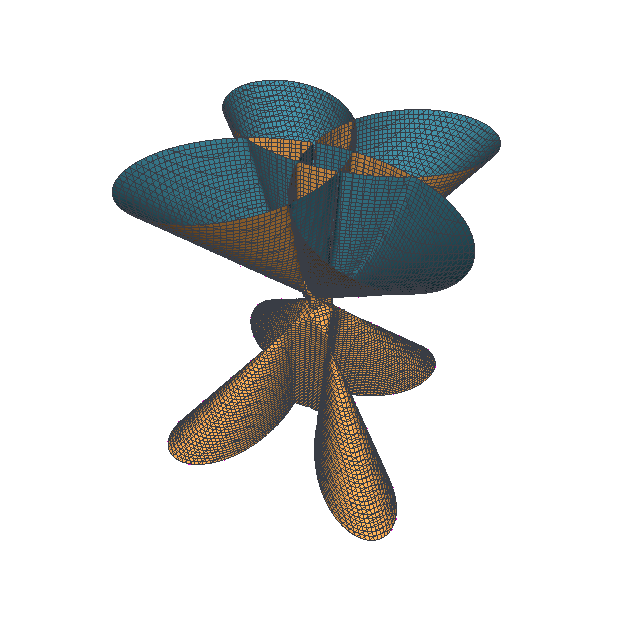}} & & &\\
		& & & & 4 reflect. \\
		$\bfx_{10}(t,s)$ & & 17 &\hspace{0.3cm}\hspace{0.3cm} 9.828\,/\,6.973\,/ &  1 axial sym.\\ & & &\hspace{0.5cm} 10.124\,/\textcolor{red}{$>90$} & 2 rotational sym. \\& & & &8 isometries \\[0.2cm]
 	\multicolumn{5}{c}{\footnotesize{$\bfx_{10}(t,s)=\left( -\frac{t^{17}-6t^{15}+6t^{11}-6t^7+6t^3-t^2-t+1}{t^2+1}, \frac{2t(t^{15}-5t^{13}-5t^{11}+t^9+t^7-5t^5-5t^3+t+1)}{t^2+1}, t(t^2+1)^3(t^8+1)\right) $}}\\
	\multicolumn{5}{c}{\footnotesize{$+s\cdot(-t^6+7t^4-7t^2+1,2t(t^4-6t^2+1),(t^2+1)^3)$}}\\	
		\bottomrule
		
	\end{tabular*}
	\caption{}
\end{table}

\restoregeometry

\section{Observations on the computation of projective equivalences.} \label{sec-projective}

Projective equivalences between $S_1,S_2$ correspond to rational mappings $f({\bf x})$ from ${\Bbb R}^3$ to ${\Bbb R}^3$ satisfying $f(S_1)=S_2$, where the components of $f$ have the form 
\begin{equation}\label{compo}
\frac{a_{i1}x+a_{i2}y+a_{i3}z+b_i}{a_{41}x+a_{42}y+a_{43}z+b_4},
\end{equation}
for $i=1,2,3$. Whenever $f$ is invertible, Theorem \ref{th-fundam} is also valid for this case, so each projective equivalence between $S_1,S_2$ has an associated mapping $\varphi(t,s)=(\varphi_1(t,s),\varphi_2(t,s))$ in the parameter space. Additionally, since projective mappings are collineations, i.e. they map lines to lines, we can argue as in the first part of 
Proposition \ref{fund} to conclude that $\varphi_1(t,s)=\psi(t)$, where $\psi(t)$ is a M\"obius transformation. However, the form of $\varphi_2(t,s)$ is not the same as in Proposition \ref{fund}, in general. Indeed, using Eq. \eqref{compo} one has that 
\[
\varphi_2(t,s)=\frac{\xi_1(t)+s\xi_2(t)}{\xi_3(t)+s\xi_4(t)},
\]
where the $\xi_j(t)$ are polynomials. As a consequence, the remaining results of Section \ref{sec-symmetries}, and in particular the form of $\varphi$ predicted by Theorem \ref{fund2}, cannot be easily generalized. Therefore, an approach analogous to the one in this paper for projective equivalences requires further work.

\section{Conclusion.} \label{sec-conclusion}

We have presented a unifying method to compute affine equivalences, isometries and symmetries of ruled rational surfaces, working directly on the rational parametric form. In order to do this, we reduce the problem to the parameter space, taking advantage of the fact that, under our hypotheses, these transformations have an associated birational transformation of the real plane whose structure can be predicted. From here, we observe that the matrix defining any affine equivalence (resp. isometry or symmetry) between the surfaces in fact defines a projective equivalence between two projective curves corresponding to the directions of the rulings of the surfaces. Thus, we take advantage of the ideas in \cite{HJ18}, where projective equivalences of curves in any dimension are considered, to solve our problem. In the case of isometries or symmetries, we have extra conditions coming from the fact that orthogonal transformations preserve norms. The algorithm is efficient, as shown in several nontrivial examples. For projective equivalences, we only provide a small hint on the form of the birational transformation of the plane behind such equivalences: giving a complete description, and generalizing the method to also cover these equivalences, requires more effort.

\end{document}